\documentclass[11pt,a4paper]{article}
\usepackage{amssymb,latexsym}
\usepackage{bm, amsmath, amssymb, amsthm}

\def\<{\langle}
\def\>{\rangle}
\def\eps{\varepsilon}
\def\RR{\mathbb{R}}

\def\calf{\mathcal{F}}
\def\tr{\operatorname{Tr\,}}
\def\id{\operatorname{id\,}}
\def\Div{\operatorname{div}}

\def\Ric{\operatorname{Ric}}
\def\vol{\operatorname{vol}}
\newcommand{\Sm}{\,{\mbox{\rm S}}}
\newcommand\const{\operatorname{const}}
\def\eq{\hspace*{-1.5mm}&=&\hspace*{-1.5mm}}

\def\minus{\hspace*{-1.5mm}&-&\hspace*{-1.5mm}}
\def\dt{\partial_t}

\newtheorem{thm}{Theorem}[section]
\newtheorem{cor}{Corollary}[section]
\newtheorem{lem}{Lemma}[section]
\newtheorem{prop}{Proposition}[section]
\newtheorem{example}{Example}[section]
\newtheorem{defn}{Definition}[section]
\newtheorem{rem}{Remark}[section]

\title{Variations of the total mixed scalar curvature of a distribution}

\author{
       Vladimir Rovenski\footnote{Mathematical Department, University of Haifa, Mount Carmel, Haifa, 31905, Israel
       \newline e-mail: {\tt vrovenski@univ.haifa.ac.il}}\ \
        \ and \
       Tomasz Zawadzki
       \footnotemark[1]
       \footnote{Faculty of Mathematics and Computer Science,
        University of \L\'{o}dz, ul. 22 Banach Str., 90-238  \L\'{o}d\'{z}, Poland
       \newline e-mail: {\tt zawadzki@math.uni.lodz.pl}}
       }

\begin{document}

\date{}

\maketitle

\begin{abstract}
We examine the total mixed scalar curvature of a fixed distribution as a functional of a pseudo-Riemannian metric. We develop variational formulas for quantities of extrinsic geometry of the distribution to find the critical points of this action. Together with the arbitrary variations of the metric, we consider also variations that preserve the volume of the manifold or partially preserve the metric (e.g. on the distribution). For each of those cases we obtain the Euler-Lagrange equation and its several solutions. The examples of critical metrics that we find are related to various fields of geometry such as contact and 3-Sasakian manifolds, geodesic Riemannian flows and codimension-one foliations of interesting geometric properties (e.g. totally umbilical, minimal).
\end{abstract}



\tableofcontents

\section*{Introduction}\label{sec:intro}

Distributions on manifolds appear in various situations - e.g. as fields of tangent planes of foliations or kernels of differential forms. When the metric of a pseudo-Riemannian manifold is non-degenerate on a distribution, it defines a pseudo-Riemannian almost-product structure - i.e., a pair of orthogonal, complementary distributions that span the tangent bundle \cite{g1967}. The mixed scalar curvature is one of the simplest curvature invariants of a pseudo-Riemannian almost-product structure. It is defined as a sum of sectional curvatures of planes that non-trivially intersect with both of the distributions (we give the exact definition and formula \eqref{eq-wal2} later). Its investigation led to multiple results regarding the existence of vector fields \cite{britowalczak}, foliations \cite{wa1} and submersions \cite{Zawadzki} of interesting geometric properties.

On a manifold with a distribution, we can define the total mixed scalar curvature as a functional on the space of all pseudo-Riemannian metrics that are non-degenerate on the distribution. Every such metric yields a pseudo-Riemannian almost-product structure and hence the mixed scalar curvature of our distribution (and of its metric-dependent orthogonal complement). Since we consider also non-compact manifolds, we assume that the total mixed scalar curvature is in fact an integral of the mixed scalar curvature over a sufficiently large, relatively compact set. When viewed as a functional of the metric, the total mixed scalar curvature may be considered an analogue of the Einstein-Hilbert action \cite{bdrs}.

The goal of this paper is to examine metrics critical for the total mixed scalar curvature with respect to different kinds of variations of metric. Apart from varying among all the metrics that are non-degenerate on the distribution, we shall also restrict to the case when the varying metric remains fixed on the distribution, and the ``complementary" case when metric varies only on the distribution -- preserving its orthogonal complement and the metric on it. Moreover, in analogy to the Einstein-Hilbert action, all variations will be considered in two versions: with and without the additional requirement that the metrics preserve the volume of the set over which we integrate \cite{besse}. Surprisingly, all the formulas that we obtain can be presented in a compact way, where the type of variation directly translates only to the values, or the form, of two terms. 

The first part of the paper starts with all necessary definitions, continues to development of variation formulas for several geometric quantities and culminates in the formulation of the Euler-Lagrange equation for the total mixed scalar curvature. The equation we obtain is difficult to solve in full generality, although it can be related to some curvatures previously described in literature \cite{r2010}. It can be decomposed into three independent parts, two of them being symmetric with respect to interchanging the given distribution and its orthogonal complement. Those two equations are also the same as the ones obtained for more restrictive, adapted variations considered in \cite{rz-1}. It is worth noting that the variation formulas for geometric quantities, that we obtain along the way to the Euler-Lagrange equation, can be of use also for many other functionals.

The second part of the paper is dedicated to examples of metrics critical for the total mixed scalar curvature. Our functional explicitly depends on the domain $\Omega$ of the integration of the mixed scalar curvature, but there exist metrics that are critical for all relatively compact $\Omega$ -- and these are of our primary interest. In Section \ref{sectionFlows} we consider the case when the fixed distribution is one-dimensional, i.e. tangent to the flowlines of a unit vector field (in 4-dimensional, general relativity setting this case corresponds to the one examined in \cite{bdrs}). We rephrase the Euler-Lagrange equation, and examine it in the case of geodesic Riemannian flows, comparing the results for different types of variations. Then, in Section \ref{SectionContact}, we consider the distribution spanned by the Reeb field on a contact manifold, which allows us to give an interpretation of some geometric quantities that appear in the Euler-Lagrange equation. In particular, we show that the metrics of $K$-contact structures are critical with respect to all variations that fix the volume and partially preserve the metric (either on the distribution or everywhere else except it), thus generalizing a result from \cite{b2010}. As a different application of the variation formulas obtained earlier, we also examine a measure of non-integrability of the orthogonal complement of the distribution, showing that for Reeb field's flowlines contact metric structures yield critical metrics. The results we obtain are then generalized, in Section \ref{section3Sasakian}, to the setting of 3-Sasakian manifolds. Finally, in Section \ref{sectionFoliations}, we consider situation ``dual" to the one from Section \ref{sectionFlows} - fixing a distribution tangent to a codimension-one foliation. Then, with the assumption of a special coordinate system, the Euler-Lagrange equations can be in some cases explicitly solved.  

\section{Main results}\label{sec:mixed-action}

\subsection{Preliminaries}\label{subsec:mixedEH}

This section recalls necessary definitions of some functions and tensors, see \cite{rz-1}, and introduces several new notions related to geometry of pseudo-Riemannian almost-product manifolds.

Let ${\rm Sym}^2(M)$ be the space of all symmetric $(0,2)$-tensors tangent to~$M$.
A \textit{pseudo-Riemannian metric} of index $q$ on $M$ is an element $g\in{\rm Sym}^2(M)$
such that each $g_x\ (x\in M)$ is a {non-degenerate bilinear form of index} $q$ on the tangent space $T_xM$.
For~$q=0$ (i.e., $g_x$ is positive definite) $g$ is a Riemannian metric and for $q=1$ it is called a Lorentz metric.
Let~${\rm Riem}(M)\subset{\rm Sym}^2(M)$ be the subspace of pseudo-Riemannian metrics of given signature.

Let $R(X,Y)=\nabla_Y\nabla_X-\nabla_X\nabla_Y+\nabla_{[X,Y]}$ be the~curvature tensor of the Levi-Civita connection $\nabla$ of $g$.
At~a point $x\in M$, a 2-dimensional linear subspace $X\wedge Y$ (called a plane section) of $T_xM$
is \textit{non-degenerate} if $W(X,Y):=g(X,X)\,g(Y,Y)-g(X,Y)\,g(X,Y)\ne0$.
For such section at $x$, the {sectional curvature} is the number $K(X\wedge Y)=g(R(X,Y)X, Y)/W(X,Y)$.

A subbundle $\widetilde{\cal D}\subset TM$ (called a distribution) is \textit{non-degenerate},
if $g_x$ is non-degenerate on $\widetilde{\cal D}_x\subset T_x M$ for every $x\in M$;
in this case, the orthogonal complement of~$\widetilde{\cal D}$, which will be denoted by ${\cal D}$, is also non-degenerate \cite{O'Neill} and we have $\widetilde{\cal D}_x\cap\,{\cal D}_x=0$, $\widetilde{\cal D}_x\oplus\,{\cal D}_x=T_xM$ for all $x \in M$. A connected manifold $M^{n+p}$ with a pseudo-Riemannian metric $g$ and a pair
of complementary orthogonal non-degene\-rate distributions $\widetilde{\cal D}$ and ${\cal D}$ of ranks
$\dim\widetilde{\cal D}_x=n$ and $\dim{\cal D}_x=p$ for every $x\in M$ is called a \textit{pseudo-Riemannian almost-product structure} on $M$, see \cite{g1967}. Such $(M, \widetilde{\cal D}, {\cal D}, g)$ is also sometimes called a \textit{pseudo-Riemannian almost-product manifold}.

Let~$\mathfrak{X}_M$ (resp. $\mathfrak{X}_{\cal D}$ and $\mathfrak{X}_{\widetilde{\cal D}}$)
be the module over $C^\infty(M)$ of all vector fields on $M$ (resp. on ${\cal D}$ and~$\widetilde{\cal D}$).
The~following convention is adopted for the range of~indices:
\[
 a,b,c\ldots{\in}\{1\ldots n\},\ i,j,k\ldots{\in}\{1\ldots p\}.
\]
The ``musical" isomorphisms $\sharp$ and $\flat$ will be used for rank one and symmetric rank 2 tensors.
For~example, if $\omega \in T^1_0 M$ is a 1-form and $X,Y\in {\mathfrak X}_M$ then
$\omega(Y)=g(\omega^\sharp,Y)$ and $X^\flat(Y) =g(X,Y)$.
For $(0,2)$-tensors $A$ and $B$ we have
 $\<A, B\> =\tr_g(A^\sharp B^\sharp)=\<A^\sharp, B^\sharp\>$.

The~sectional curvature $K(X\wedge Y)$ is called \textit{mixed} if $X\in\widetilde{\cal D}$ and $Y\in{\cal D}$.
Let $\{E_a,\,{\cal E}_{i}\}$ be a local orthonormal frame \textit{adapted} to $(\widetilde{\cal D},\,{\cal D})$, i.e.
\begin{eqnarray*}
&& E_a \in \widetilde{\cal D}, \quad a = 1, \ldots, n , \\
&& {\cal E}_{i} \in {\cal D}, \quad i=1 , \ldots , p
\end{eqnarray*}
and let $\epsilon_i=g({\cal E}_{i},{\cal E}_{i}),\ \epsilon_a=g(E_a,E_a) $. We have $| \epsilon_i | = | \epsilon_a | = 1$ and $W(E_a, {\cal E}_{i})=\epsilon_a \epsilon_i\ne0$. The function on $M$,
\begin{equation}\label{eq-wal2}
 \Sm_{\rm mix} =\sum\nolimits_{\,a,i} K(E_a\wedge {\cal E}_{i})
 =\sum\nolimits_{\,a,i}\epsilon_a \epsilon_i\,g(R(E_a, {\cal E}_{i})E_a,\, {\cal E}_{i})\,
\end{equation}
is called the \textit{mixed scalar curvature}, see~\cite{wa1}, and does not depend on the choice of the adapted orthonormal frame.
 If~a distribution is spanned by a unit vector field $N$, i.e., $g(N,N)=\epsilon_N\in\{-1,1\}$,
 then $\Sm_{\rm mix} = \epsilon_N\Ric_{N}$, where $\Ric_{N}$ is the Ricci curvature in the $N$-direction.

To compute $\Sm_{\rm mix}$ on $(M,g)$ we only need to fix one of the distributions, say $\widetilde{\cal D}$,
then we obtain the second distribution as its $g$-orthogonal complement and the function \eqref{eq-wal2} is well defined. Given a pair $(M,\widetilde{\cal D})$ of a manifold and a distribution, we shall study pseudo-Riemannian structures non-degenerate on $\widetilde{\cal D}$ and critical for the functional
\begin{equation}\label{E-Jmix}
 J_{\rm mix,\widetilde{\cal D},\Omega} :\ g \mapsto \int_{\Omega} \Sm_{\rm mix}(g)\, {\rm d} \vol_g ,
\end{equation}
where $\Omega$ in (\ref{E-Jmix}) is a relatively compact domain of $M$ (and $\Omega=M$ when $M$ is closed), containing supports of variations of the metric.

The Euler-Lagrange equation for \eqref{E-Jmix}, that we shall obtain later, is expressed in terms of extrinsic geometry of the distribution $\widetilde{\cal D}$ and its orthogonal complement ${\cal D}$. In order to understand it, we shall define several notions on a pseudo-Riemannian almost-product manifold $(M, \widetilde{\cal D}, {\cal D}, g)$.

For every $X\in\mathfrak{X}_M$ we have $X=\widetilde{X} + X^\perp$, where $\widetilde{X} \equiv X^\top$ is the $\widetilde{\cal D}$-component of $X$  (resp. $X^\perp$ the  ${\cal D}$-component of $X$) with respect to $g$. We define $g^\perp$ and $g^\top$ by
\[
g^\perp (X,Y) = g( X^\perp , Y^\perp ) ,\quad g^\top (X,Y) = g(X^\top , Y^\top) , \qquad (X,Y \in \mathfrak{X}_M ).
\]
The 
symmetric $(0,2)$-tensor ${r}_{\,\cal D}$, given~by
\begin{equation}\label{E-Rictop2}
 {r}_{{\cal D}}(X,Y) = \sum\nolimits_{a} \epsilon_a\,g(R(E_a, \, X^\perp)E_a, \, Y^\perp), \qquad ( X,Y\in \mathfrak{X}_M),
\end{equation}
is referred to as the \textit{partial Ricci tensor} adapted for $\cal D$, see \cite{rz-1}.
In particular, by~\eqref{eq-wal2},
\begin{equation}\label{E-Ric-Sm}
 \tr_{g}{r}_{\cal D}=\Sm_{\rm mix}(g).
\end{equation}
Note that the {partial Ricci curvature} $r_{\cal D}(X,X)$ in the direction of
a unit vector $X\in{\cal D}$ is the sum of sectional curvatures over all mixed planes containing $X$.

Let~$T, h:\widetilde{\cal D}\times \widetilde{\cal D}\to{\cal D}$
be the integrability tensor and the second fundamental form of $\widetilde{\cal D}$,
\begin{eqnarray*}
  T(X,Y)=(1/2)\,[X,\,Y]^\perp,\quad h(X,Y) \eq (1/2)\,(\nabla_X Y+\nabla_Y X)^\perp \quad (X, Y \in {\widetilde{\cal D}}) , \\
 \tilde T(X,Y)=(1/2)\,[X,\,Y]^\top,\quad \tilde h(X,Y) \eq (1/2)\,(\nabla_X Y+\nabla_Y X)^\top \quad (X,Y \in {\cal D}).
\end{eqnarray*}
Using an orthonormal adapted frame,
one may find the formulae
\begin{eqnarray*}
 \<\tilde h,\tilde h\>
 =\!\sum\nolimits_{\,i,j}\epsilon_i\epsilon_j\,g(\tilde h({\cal E}_i,{\cal E}_j),\tilde h({\cal E}_i,{\cal E}_j)),\quad
 \<h,h\>=\!\sum\nolimits_{\,a,b}\epsilon_a\epsilon_b\,g(h({E}_a,{E}_b),h({E}_a,{E}_b)), \\ 
 \<\tilde T,\tilde T\>=\!\sum\nolimits_{\,i,j}\epsilon_i\epsilon_j\,
 g(\tilde T({\cal E}_i,{\cal E}_j),\tilde T({\cal E}_i,{\cal E}_j)),\quad
 \<T,T\>=\!\sum\nolimits_{\,a,b}\epsilon_a\epsilon_b\,g(T({E}_a,{E}_b),T({E}_a,{E}_b)).
\end{eqnarray*}
The mean curvature vector fields of $\widetilde{\cal D}$
and ${\cal D}$ are, respectively,
\[
 H=\tr_{g} h=\sum\nolimits_a\epsilon_a h(E_a,E_a)
,\quad
 \tilde H=\tr_{g}\tilde h=\sum\nolimits_i\epsilon_i \tilde h({\cal E}_i,{\cal E}_i).
\]
 A~distribution $\widetilde{\cal D}$
is called \textit{totally umbilical}, \textit{harmonic}, or \textit{totally geodesic}, if
 $h=\frac1nH\,{g^\top},\ H =0$, or $h=0$, respectively.


The Weingarten operator $A_Z$ of $\widetilde{\cal D}$ with respect to $Z\in{\cal D}$,
and the operator $T^\sharp_Z$ are defined~by
\[
 g(A_Z(X),Y)= g(h(X,Y),Z),\quad g(T^\sharp_Z(X),Y)=g(T(X,Y),Z) , \qquad (X,Y \in \widetilde{\cal D}) .
\]
Similarly, we define for $N \in \widetilde{\cal D}$
\[
g( {\tilde A}_N (X), Y ) = g({\tilde h}(X,Y),N),\quad g({\tilde T}^\sharp_N(X),Y)=g({\tilde T}(X,Y),N) , \qquad (X,Y \in {\cal D}).
\]
For the local orthonormal frame $\{E_i,\,{\cal E}_a\}_{i\le p,\,a\le n}$
(adapted to the distributions)
we use the following convention for various $(1,1)$-tensors:
 ${\tilde T}^{\sharp}_a := {\tilde T}^{\sharp}_{ E_a },\ A_i := A_{ {\cal E}_i }$, etc.

The Divergence Theorem states that $\int_{M} (\Div\xi)\,{\rm d}\vol_g =0$, when $M$ is closed;
this is also true if $M$ is open and $\xi\in\mathfrak{X}_M$ is supported in a relatively compact domain
$\Omega\subset M$.
The~${\cal D}^\bot$-\textit{divergence} of $\xi$ is defined by
$\Div^\perp \xi=\sum\nolimits_{i} \epsilon_i\,g(\nabla_{{\cal E}_i}\,\xi, {\cal E}_i)$ , similarly $\widetilde{\Div} \, \xi=\sum\nolimits_{a} \epsilon_a\,g(\nabla_{ E_a }\,\xi, E_a )$.
 Thus, the divergence of $\xi$ is
\[
 \Div\xi=\tr(\nabla \xi) = \Div^\perp\xi + \widetilde{\Div}\,\xi.
\]
Recall that for a vector field $X\in\mathfrak{X}_{\cal D}$ we have
\begin{eqnarray}\label{E-divN}
 {\Div}^\bot X = \Div X+g(X,\,H).
\end{eqnarray}
Indeed, using $H=\sum\nolimits_{\,a\le n} \epsilon_a\,h(E_{a}, E_{a})$
and $g(X,\,E_{a})=0$, one derives~(\ref{E-divN}):
\begin{equation*}
 \Div X-{\Div}^\bot X =\sum\nolimits_{a} \epsilon_a\,g(\nabla_{E_a} X,\,E_a)
 =-\sum\nolimits_{a}\epsilon_a\,g(h(E_{a}, E_{a}), X) = -g(X,\,H).
\end{equation*}
For a $(1,2)$-tensor $P$ define a $(0,2)$-tensor ${\Div}^\bot P$ by
\[
 ({\Div}^\bot P)(X,Y) = \sum\nolimits_i \epsilon_i\,g((\nabla_{{\cal E}_i}\,P)(X,Y), {\cal E}_i) , \quad (X,Y \in \widetilde{\cal D}).
\]
Then the divergence of $P$ is $(\Div\,P)(X,Y) = \widetilde{\Div}\,P +{\Div}^\bot P$.
For a~${\cal D}$-valued $(1,2)$-tensor $P$, similarly to \eqref{E-divN}, we have
$\sum\nolimits_a \epsilon_a\,g((\nabla_{E_a}\,P)(X,Y), E_a)=-g(P(X,Y), H)$ and
\begin{eqnarray}\label{E-divP}
 {\Div}^\bot P = \Div P+\<P,\,H\>\,,
\end{eqnarray}
where $\<P,\,H\>(X,Y):=g(P(X,Y),\,H)$ is a $(0,2)$-tensor.
For example,
 $\Div^\perp h = \Div h+\<h,\,H\>$.

For any function $f$ on $M$, we introduce the following notation of the projections of its gradient onto distributions $\widetilde{\cal D}$ and $\cal D$:
\[
\nabla^{\top} f \equiv {\widetilde \nabla} f := (\nabla f)^{\top} , \quad \nabla^{\perp} f := ( \nabla f)^{\perp} .
\]
The $\widetilde{\cal D}$-{Laplacian} of a function $f$ is given by the formula
$\widetilde\Delta\,f=\widetilde{\Div}\,(\widetilde\nabla\,f)$.


 The ${\cal D}$-\textit{deformation tensor} ${\rm Def}_{\cal D}\,Z$ of a vector field $Z$ (e.g. $Z=H$) is the symmetric part of $\nabla Z$
restricted to~${\cal D}$,
\[
 2\,{\rm Def}_{\cal D}\,Z(X,Y)=g(\nabla_X Z, Y) +g(\nabla_Y Z, X),\quad ( X,Y\in \mathfrak{X}_{\cal D} ).
\]
As in \cite{rz-1}, we define self-adjoint $(1,1)$-tensors:
${\cal A}:=\sum\nolimits_{\,i}\epsilon_i A_{i}^2\,$, called the \textit{Casorati operator} of ${\cal D}$,
and
 ${\cal T}:=\sum\nolimits_{\,i}\epsilon_i(T_{i}^\sharp)^2$.
We also define the
symmetric $(0,2)$-tensor $\Psi$ by the identity
\begin{eqnarray*}
 \Psi(X,Y) \eq \tr (A_Y A_X+T^\sharp_Y T^\sharp_X),
 \quad (X,Y\in\mathfrak{X}_{\cal D})\,.
\end{eqnarray*}

The partial Ricci tensor can be presented in terms of the extrinsic geometry.

\begin{prop}[see \cite{rz-1}]\label{L-CC-riccati}
Let $g\in{\rm Riem}(M,\,\widetilde{\cal D},\,{\cal D})$. Then the following identities hold:
\begin{eqnarray}\label{E-genRicN}
 r_{{\cal D}} \eq \Div\tilde h +\<\tilde h,\,\tilde H\>
  -\widetilde{\cal A}^\flat -\widetilde{\cal T}^\flat
 -\Psi +{\rm Def}_{\cal D}\,H 
\end{eqnarray}
\end{prop}
We define the extrinsic scalar curvatures of $\widetilde{\cal D}$ and $\cal D$ by
\[
\Sm_{\,\rm ex} = g(H,H)-\<h,h\> , \quad \widetilde{\Sm}_{\,\rm ex} = g( {\tilde H} , {\tilde H} )-\< {\tilde h},{\tilde h}\>,
\]
respectively.
Tracing \eqref{E-genRicN} over ${\cal D}$ and applying \eqref{E-Ric-Sm} and the equalities
\begin{eqnarray*}
 \tr{\cal A}\eq\<h,h\>,\quad \tr{\cal T} = -\<T,T\>,\\
 \tr_{g}\Psi \eq \tr({\cal A}+{\cal T}) = \<h,h\> - \<T,T\>,\\
 \tr_{g}\,({\Div}\,h) \eq \Div H,\quad
 \tr_{g}\,({\rm Def}_{\cal D}\,H) = \Div H +g(H,H)
\end{eqnarray*}
yields the formula (see also \cite{wa1})
\begin{equation}\label{eq-ran-ex}
 \Sm_{\rm mix} = \Sm_{\,\rm ex} +\widetilde\Sm_{\,\rm ex} +\<T,T\> +\<\tilde T,\tilde T\> + \Div(H+\tilde H)\,,
\end{equation}
which shows how $\Sm_{\rm mix}$ is built of invariants of the extrinsic geometry of the distributions.

We define the following $(1,2)$-tensors on $(M, \widetilde{\cal D} , {\cal D}, g)$ 
for all $X,Y,Z \in \mathfrak{X}_M$:
\begin{eqnarray*}
 \alpha(X,Y) \eq \frac{1}{2}\,(A_{X^{\perp}} (Y^{\top}) + A_{Y^{\perp}} (X^{\top})), \quad
 {\tilde\alpha}(X,Y)=\frac{1}{2}\,\big({\tilde A}_{X^{\top}}(Y^{\perp}) + {\tilde A}_{Y^{\top}}(X^{\perp})\big),\\
 \theta(X,Y) \eq \frac{1}{2}\,(T^{\sharp}_{X^{\perp}}(Y^{\top}) + T^{\sharp}_{Y^{\perp}}(X^{\top})), \quad
 {\tilde\theta}(X,Y) = \frac{1}{2}\,\big({\tilde T}^{\sharp}_{X^{\top}} (Y^{\perp})
 + {\tilde T}^{\sharp}_{Y^{\top}}(X^{\perp})\big), \\
 {\tilde\delta}_{Z}(X,Y) \eq \frac{1}{2}\,\big(g(\nabla_{X^{\top}} Z,\, Y^{\perp}) + g(\nabla_{Y^{\top}} Z,\, X^{\perp})\big).
\end{eqnarray*}
For any $(0,2)$-tensors $P,Q$ and $S$ on $TM$, we define a tensor $\Lambda_{P,Q}$ by
\[
 \<\Lambda_{P,Q}, S\> =  \sum\nolimits_{\,\lambda, \mu } \epsilon_\lambda \epsilon_\mu
 [S(P(e_{\lambda}, e_{\mu}), Q( e_{\lambda}, e_{\mu})) +
 S(Q(e_{\lambda}, e_{\mu}), P( e_{\lambda}, e_{\mu}))],
\]
where $\{e_{\lambda}\}$ is a full orthonormal basis of $TM$ and $\epsilon_\lambda = g(e_{\lambda}, e_{\lambda})\in\{-1,1\}$.

We also use symmetric $(0,2)$-tensors $\Phi_h$ and $\Phi_{T}$ defined as in \cite{rz-1}, i.e., for any symmetric $(0,2)$-tensor $S$ we have
\begin{eqnarray*}
 \<\Phi_h,\ S\> \eq S(H,\,H) -\sum\nolimits_{\,a,\,b} \epsilon_a\,\epsilon_b\,S(h(E_a,E_b), h(E_a,E_b)),\\
 \<\Phi_T,\ S\> \eq -\sum\nolimits_{\,a,\,b} \epsilon_a\,\epsilon_b\,S(T(E_a,E_b), T(E_a,E_b)).
\end{eqnarray*}
Note that $\Phi_T = -\frac12\,\Lambda_{T,T}$ 
and $\Phi_h = H^\flat \otimes H^\flat - \frac{1}{2}\Lambda_{h,h}$.

We define a self-adjoint $(1,1)$-tensor (with zero trace)
\[
 {\cal K} =\sum\nolimits_{\,i} \epsilon_i\,[T^\sharp_i, A_i] = \sum\nolimits_{\,i} \epsilon_i\,(T^\sharp_i A_i - A_i\,T^\sharp_i)
\]
and its counterpart $\widetilde{\cal K} =\sum\nolimits_{\,a} \epsilon_a\,[{\tilde T}^\sharp_a, {\tilde A}_a]$. It is easy to see that all the above tensors defined with the use of an adapted orthonormal frame in fact do not depend on the choice of such frame.

\begin{rem}[see \cite{rz-1}]\rm \label{Phihzero}
 If $g$ is definite on $\widetilde{\cal D}$ then $\Phi_{h} =0$ if and only if one of the following point-wise conditions holds:
\[
 (i)~h=0; \quad
 (ii)~H\ne0,\ \Sm_{\,\rm ex}=0 \ \mbox{\rm and the image of}\ h\ \mbox{\rm is spanned by} \ H\,.
\]
If ${\cal D}$ is integrable then $\tilde T^\sharp_a = 0$ for all $a=1, \ldots, n$,
hence $\tilde{\cal K} :=\sum\nolimits_{a} \epsilon_a\,[\tilde T^\sharp_a , \tilde A_a]=0$.
If ${\cal D}$ is totally umbilical, then every operator $\tilde A_a$ is a multiple of identity and $\tilde{\cal K}$ vanishes as well.
\end{rem}

\begin{rem}\rm Let $B$ be a symmetric (0,2)-tensor. The following computations will be used to obtain variation formulas: 
\begin{eqnarray*}
 \<\,\<\alpha, {\tilde H}\>, B\> \eq \sum\nolimits_{\,a,i} \epsilon_a \epsilon_i  g(\alpha(E_a, {\cal E}_i),{\tilde H}) B(E_a,{\cal E}_i)
 +\sum\nolimits_{\,a,i} \epsilon_a \epsilon_i g(\alpha({\cal E}_i, E_a), {\tilde H}) B({\cal E}_i, E_a) \\
 \eq 2 \sum\nolimits_{\,a,i} \epsilon_a \epsilon_i g(\alpha(E_a, {\cal E}_i), {\tilde H}) B(E_a,{\cal E}_i) \\
 \eq \sum\nolimits_{\,a,i} \epsilon_a \epsilon_i g(A_{i}(E_a), {\tilde H}) B(E_a,{\cal E}_i) ,
\end{eqnarray*}
\[
 \<\Lambda_{\alpha, \theta}, B\> = \sum\nolimits_{\,a,i}\epsilon_a\epsilon_i B(A_{i}(E_a), T^{\sharp}_i(E_a)) .
\]
Later we will also use the fact that for $X \in \widetilde{\cal D}$, $N \in {\cal D}$ we have
\[
\Lambda_{\alpha , {\tilde \theta} } (X,N) = \frac{1}{2} \sum_{a,i} \epsilon_a \epsilon_i g(X, A_{ i } E_a ) g( N , {\tilde T}^\sharp_{ a } {\cal E}_i ) .
\]
Similar formulas can be obtained for $\Lambda_{\alpha,  {\tilde \alpha}}$, $\Lambda_{\theta,  {\tilde \alpha}}$, etc.

\end{rem}

\subsection{Variation formulas}
\label{sec:prel}

Let $(M, \widetilde{\cal D},g)$ be a manifold with distribution and a pseudo-Riemannian metric $g$.
We consider smooth $1$-parameter variations $\{g_t\in{\rm Riem}(M):\ |t|<\eps\}$ of the metric $g_0 = g$.
We assume that the~induced infinitesimal variations, represented by a symmetric $(0,2)$-tensor ${B}_t\equiv\partial g_t/\partial t$,
are supported in a relatively compact domain $\Omega$ in $M$. We adopt the notations
\begin{equation}\label{E-Sdtg-2}
 \partial_t \equiv \partial/\partial t,\quad {B} \equiv {\dt g_t}_{\,|\,t=0}.
\end{equation}
Since $B$ is symmetric,
for any  $(0,2)$-tensor $C$ we have
 $\<C,\,B\>=\<{\rm Sym}(C),\,B\>$.
We denote by ${\cal D}(t)$ the $g_t$-orthogonal complement of $\widetilde{\cal D}$.

\begin{defn}\rm
Let $\widetilde{\cal D}$ be a distribution on $(M, g)$.
A family of metrics $\{g_t\in{\rm Riem}(M):\ |t|<\eps\}$
such that $g_0 =g$ and for all~$|t|<\eps$:
\begin{equation}\label{e:Var}
 g_{t}(X,Y)=g(X,Y),\quad (X,Y\in\widetilde{\cal D}),
\end{equation}
will be called \textit{$g^{\perp}$-variation}. For $g^{\perp}$-variations the metric on $\widetilde{\cal D}$ is preserved.
\end{defn}

\begin{defn}\rm
Let $\widetilde{\cal D}$ be a distribution on $(M, g)$ and let ${\cal D}$ be its $g$-orthogonal complement.
A family of metrics $\{g_t\in{\rm Riem}(M):\ |t|<\eps\}$ such that $g_0 =g$, for all~$|t|<\eps$ the distributions $\widetilde{\cal D}$ and ${\cal D}$ remain orthogonal and
\begin{equation}\label{e:Vartop}
 g_{t}(X,Y)=g(X,Y), \quad (X,Y\in {\cal D}),
\end{equation}
will be called \textit{${{g^\top}}$-variation}. For ${{g^\top}}$-variations only the metric on $\widetilde{\cal D}$ changes.
\end{defn}

We will now relate the variations defined above to arbitrary variations of $g$.
Let ${\cal D} = {\cal D}(0)$ be the $g$-orthogonal complement of $\widetilde{\cal D}$. While it may not be $g_t$-orthogonal for $t>0$, we can assume that 
the distributions $\widetilde{\cal D}$ and ${\cal D}$ span the tangent bundle.
For any $X \in TM$, let $X_{\widetilde{\cal D}}$ denote the $g$-orthogonal projection of $X$ onto $\widetilde{\cal D}$ and let $X_{\cal D}$ denote the $g$-orthogonal projection of $X$ onto ${\cal D}$. Then,
given $g\in{\rm Riem}(M)$, we have $g_t = g_{t |\, {\cal D} \times {\cal D}} + {g}_{t \,|\,{\cal D}\times\widetilde{\cal D}} + {g}_{t\,|\,\widetilde{\cal D} \times {\cal D} } + {g}_{t\,|\,\widetilde{\cal D}\times\widetilde{\cal D}}$,
where 
\begin{eqnarray*}
&& {g}_{t\,|\,\widetilde{\cal D}\times\widetilde{\cal D}} (X,Y)= g_t(X_{\widetilde{\cal D}},Y_{\widetilde{\cal D}}) , \quad  g_{t |\,{\cal D}\times{\cal D}} (X,Y) = g_t(X_{\cal D},Y_{\cal D}) ,  \\
&& {g}_{t\,|\,{\cal D}\times\widetilde{\cal D}} (X,Y)= g_t(X_{\cal D},Y_{\widetilde{\cal D}}), \quad {g}_{t\,|\,\widetilde{\cal D} \times {\cal D} } (X,Y) = g_t(X_{\widetilde{\cal D}},Y_{\cal D}),
\end{eqnarray*}
and we can present $g_t$ in the following form:
\[
 g_t 
  = \bigg(\begin{array}{cc}
   g_{t\,|\,{\cal D}\times{\cal D}} & {g}_{t\,|\,{\cal D}\times\widetilde{\cal D}} \\
   {g}_{t\,|\,\widetilde{\cal D}\times{\cal D}} & {g}_{t\,|\,\widetilde{\cal D}\times\widetilde{\cal D}}
 \end{array}\bigg) .
\]

Similarly,
$B_t=B_t^\perp + {B}_t^{_/} + {\tilde B}_t$, where
$B_t^\perp= \dt g_{t |\,{\cal D}\times{\cal D}} ,\ {\tilde B}_t = \dt {g}_{t\,|\,\widetilde{\cal D}\times\widetilde{\cal D}} $ and ${B}_t^{_/}= \dt {g}_{t\,|\,{\cal D}\times\widetilde{\cal D}} + \dt {g}_{t\,|\,\widetilde{\cal D} \times {\cal D} }$.
For $g^\perp$-variations $g_t = g_{t |\,{\cal D}\times{\cal D}} + {g}_{t\,|\,{\cal D}\times\widetilde{\cal D}} + {g}_{t\,|\,\widetilde{\cal D} \times {\cal D} } +  {g}_{0\,|\,\widetilde{\cal D}\times\widetilde{\cal D}}$ and for
${g^\top}$-variations $g_t = g_{0 |\,{\cal D}\times{\cal D}} + {g}_{0\,|\,{\cal D}\times\widetilde{\cal D}} + {g}_{0\,|\,\widetilde{\cal D} \times {\cal D} } +  {g}_{t\,|\,\widetilde{\cal D}\times\widetilde{\cal D}}$ we have, respectively,
\[
 B_t = B_t^\perp + {B}_t^{_/}
 =\bigg(\begin{array}{cc}
   B^\perp_{ t \,|\,{\cal D}\times{\cal D}} &
   {B}^{_/}_{ t \,|\,{\cal D}\times\widetilde{\cal D}} \\
   {B}^{_/}_{ t \,|\,\widetilde{\cal D}\times{\cal D}} & 0
 \end{array}\bigg),
 \quad
 B_t = {\tilde B}_t
 =\bigg(\begin{array}{cc}
   0 & 0 \\ 0 & {\tilde B}_{ t \,|\,\widetilde{\cal D}\times\widetilde{\cal D}}
 \end{array}\bigg).
\]
By the above, the derivative $B_t$ of any variation $g_t$ can be decomposed into sum of derivatives of some $g^\perp$- and ${g^\top}$-variations.

The Levi-Civita connection $\nabla^t$ of $g_t\ (|t| < \eps)$ evolves as, see for example~\cite{topp},
\begin{equation}\label{eq2G}
 2\,g_t(\dt(\nabla^t_X\,Y), Z) = (\nabla^t_X\,{B})(Y,Z)+(\nabla^t_Y\,{B})(X,Z)-(\nabla^t_Z\,{B})(X,Y),\quad
 X,Y,Z\in\mathfrak{X}_M,
\end{equation}
where the first covariant derivative of a $(0,2)$-tensor ${B}$ is expressed~as
\begin{eqnarray*}
( \nabla_{Z}\,{B} )(Y,V)\eq Z({B}(Y,V)) -{B}(\nabla_Z Y, V)-{B}(Y,\nabla_Z V).
\end{eqnarray*}

Recall that the distribution ${\cal D}(t)$ is defined as $g_{t}$-orthogonal complement of $\widetilde{ \cal D}$. 
Let $^\top$ and $^\perp$ denote the $g_t$-orthogonal projections onto $\widetilde{\cal D}$ and ${\cal D}(t)$, respectively. Note that these projections are $t$-dependent.

\begin{lem}\label{prop-Ei-a}
Let $g_t$ be a $g^\perp$-variation of $g$ with $B_t = \dt g_t$. Let $\{E_a,\,{\cal E}_{i}\}$ be a local $(\widetilde{\cal D},\,{\cal D})$-adapted and orthonormal for $t=0$ frame, that evolves according to
 \begin{equation}\label{E-frameE}
 \dt E_a = 0,\qquad
 \dt {\cal E}_{i}=-(1/2)\, ({B}_t^\sharp({\cal E}_{i}))^{\perp} - ({B}_t^\sharp({\cal E}_{i}))^{\top}.
\end{equation}
Then, for all $\,t, $ $\{E_a(t),{\cal E}_{i}(t)\}$ is a $g_t$-orthonormal frame
adapted to $(\widetilde{\cal D},{\cal D}(t))$.
\end{lem}

\proof
Since $\dt E_a = 0$ and $E_a (0) \in \widetilde{\cal D}$, we have for $g^\perp$-variation
\[
 \dt(g_t(E_a, E_b))=0.
\]


Also,
\begin{eqnarray*}
 &&\dt(g_t(E_a, {\cal E}_i)) = (\dt g_t)(E_a(t), {\cal E}_i (t)) + g_t(\dt E_a(t), {\cal E}_i (t)) +g_t(E_a(t), \dt {\cal E}_i (t))
   \\
 &&= {B}_t(E_a(t), {\cal E}_i (t)) - \frac12\,g_t(({B}_t^\sharp({\cal E}_i (t)))^{\perp}, E_a(t)) - \,g_t(E_a(t), {B}_t^\sharp({\cal E}_i  (t))^{\top})=0.
\end{eqnarray*}
Now that we know that $g_t(E_a, {\cal E}_i) =0$, it follows that ${\cal E}_i (t) \in {\cal D}(t)$ and for any $X$ we have $g_t({\cal E}_i, X^{\top}) =0$. We can finish the proof by computing
\begin{eqnarray*}
 &&\dt(g_t({\cal E}_i, {\cal E}_j)) =  g_t(\dt {\cal E}_i(t), {\cal E}_j (t)) +g_t({\cal E}_i(t), \dt {\cal E}_j (t)) +(\dt g_t)({\cal E}_i(t), {\cal E}_j (t))  \\
 &&= {B}_t({\cal E}_i(t), {\cal E}_j (t)) - \frac12\,g_t(({B}_t^\sharp({\cal E}_i(t)))^{\perp}, {\cal E}_j (t)) - \,g_t({\cal E}_i(t), ({B}_t^\sharp {\cal E}_j  (t))^{\perp})=0.
 \qed
\end{eqnarray*}

The evolution of ${\cal D}(t)$ gives rise to the evolution of both $\widetilde{ \cal D}$- and ${\cal D}(t)$-components of any vector $X$ on $M$.

\begin{lem}\label{projections}
Let $g_t$ be a $g^\perp$-variation of $g$. Then for any $t$-dependent vector $X$ on $M$, we have:
\begin{eqnarray*}
 \dt (X^{\top}) = (\dt X)^{\top} + (B^{\sharp} (X^{\perp}))^{\top},\qquad
 \dt (X^{\perp}) = (\dt X)^{\perp} - (B^{\sharp} (X^{\perp}))^{\top}.
\end{eqnarray*}
\end{lem}

\proof
Using the frame from Lemma \ref{prop-Ei-a}, we can write
 \begin{eqnarray}
 \label{Xtop}
 && X^{\top} = \sum_a \epsilon_a g_t (X_t , E_a (t) ) E_a , \\
 \label{Xperp}
 && X^{\perp} = \sum_i \epsilon_i g_t (X_t , {\cal E}_i (t) ) {\cal E}_i (t).
 \end{eqnarray}
We have
\begin{eqnarray*}
{B}_t(E_a(t), E_b(t)) \eq (\dt g_t)(E_a(t), E_b(t))  \\
\eq \dt(g_t(E_a, E_b)) - g_t(\dt E_a(t), E_b(t)) - g_t(E_a(t), \dt E_b(t)) =0,
\end{eqnarray*}
hence, $(B^{\sharp}(X^{\top}))^{\top}=0$, which implies
\begin{equation} \label{Bsharptop}
(B^{\sharp}(X))^{\top} = (B^{\sharp} (X^{\perp}))^{\top}.
\end{equation}
The proof follows from differentiating \eqref{Xtop} and \eqref{Xperp}, using \eqref{E-frameE} and \eqref{Bsharptop}.
\qed


\begin{prop} \label{propvar1}
Let $g_t$ be a $g^\perp$-variation of $g$. Then
\begin{subequations}
\begin{eqnarray}\label{E-tildeh-gen}
 &&\hskip-9mm
 \dt\<{\tilde h}, {\tilde h}\> =  \<\Div{\tilde h} - 4\,\Lambda_{ {\tilde\alpha}, \theta }
 +\widetilde{\cal K}^\flat,\, B\> - \Div\<{\tilde h}, B\>,\\
\label{E-tildeH-gen}
 &&\hskip-9mm \dt g({\tilde H}, {\tilde H}) = \<\,(\Div{\tilde H})\,g^{\perp}
 + 4\,\<\theta, {\tilde H}\>,\, B\> -\Div ((\tr_{\cal D} B^\sharp){\tilde H}),\\
\label{E-h-gen}
 &&\hskip-9mm \dt \<h,h\> = 2 \Div(\<\alpha, B\>) -2\,\< ( \Div \alpha )_{ | \widetilde{\cal D} \times {\cal D}} + ( \Div \alpha )_{ |  {\cal D} \times \widetilde{\cal D} }  + \Lambda_{\alpha, {\tilde\alpha}
 + {\tilde \theta}} + \Phi_h ,\, B\> \nonumber \\ && -B(H,H),\\
\label{E-H-gen}
 &&\hskip-9mm \dt g(H, H) = -B(H,H) +2\,\big\<\,\<{\tilde\theta - \tilde\alpha},\, H\>
 +{\rm Sym}(H^{\flat}\otimes{\tilde H}^{\flat}) -{\tilde\delta}_H,\, B\big\> + 2 \Div (B^{\sharp} H)^{\top},\quad\\
\label{E-tildeT-gen}
 &&\hskip-9mm \dt\<{\tilde T}, {\tilde T}\>
 = 2\,\<\,\tilde{\cal T}^\flat +\Lambda_{{\tilde \theta},\theta-\alpha}- ( \Div{\tilde \theta} )_{ | \widetilde{\cal D} \times {\cal D}} - ( \Div \tilde \theta )_{ |  {\cal D} \times \widetilde{\cal D} } ,\, B\,\>
 + 2\Div \<{\tilde \theta},\,B\>,\\
\label{E-T-gen}
 &&\hskip-9mm \dt\<T, T\> = - \<\Phi_T, B\>.
\end{eqnarray}
\end{subequations}
\end{prop}

\begin{proof}
In the proof we denote by (i)$_j$ the $j$-th term of the right hand side of formula (i). Assume $\nabla^t_{{\cal E}_i}\,{\cal E}_j \in \widetilde{\cal D}$ at a point $x\in M$.
Since $B$ vanishes on $\widetilde{\cal D}\times\widetilde{\cal D}$, we obtain
\[
 B({\tilde h}({\cal E}_i,{\cal E}_j),{\tilde h}({\cal E}_i,{\cal E}_j)) =0.
\]
\textbf{Proof of} \eqref{E-tildeh-gen}. We use Lemma~\ref{projections} to compute $\dt\<\tilde h,\tilde h\>$, as a sum of 10 terms in $g(\cdot\,,E_a)$,
\begin{eqnarray}\label{E-10terms}
 &&\dt\<\tilde h,\tilde h\> = \sum\nolimits_{\,i,j}\,B(\tilde h({\cal E}_i,{\cal E}_j),\tilde h({\cal E}_i,{\cal E}_j))
 +\sum\nolimits_{\,i,j}\,g(\tilde h({\cal E}_i,{\cal E}_j), \dt((\nabla^t_{{\cal E}_i}\,{\cal E}_j +\nabla^t_{{\cal E}_j}\,{\cal E}_i )^\top))\\
\nonumber
 &&\hskip-5mm = \sum\nolimits_{\,i,j,a}\,g(\tilde h({\cal E}_i,{\cal E}_j),E_a)\,
 g\big(
 \nabla^t_{(\dt{\cal E}_i)^\top}{\cal E}_j
 +\nabla^t_{(\dt{\cal E}_j)^\top}{\cal E}_i
 +\nabla^t_{{\cal E}_i}((\dt{\cal E}_j)^\top)
 +\nabla^t_{{\cal E}_j}((\dt{\cal E}_i)^\top)\\
\nonumber
 &&\hskip-8mm
 +\,\nabla^t_{{\cal E}_i}((\dt{\cal E}_j)^\bot)
 +\nabla^t_{{\cal E}_j}((\dt{\cal E}_i)^\bot)
 +\nabla^t_{(\dt{\cal E}_i)^\bot}\,{\cal E}_j
 +\nabla^t_{(\dt{\cal E}_j)^\bot}\,{\cal E}_i
 +(\dt\nabla^t)_{{\cal E}_i}\,{\cal E}_j+(\dt\nabla^t)_{{\cal E}_j}\,{\cal E}_i,\,
 E_a\big),
\end{eqnarray}
where the last two terms \eqref{E-10terms}$_9$ and \eqref{E-10terms}$_{10}$ are equal 
and can be computed in the following way:
\begin{eqnarray*}
 &&\hskip-8mm g((\dt \nabla^t)_{{\cal E}_i}\,{\cal E}_j, E_a) = (\nabla^t_{{\cal E}_i}\,B)({\cal E}_j, E_a)
 +(\nabla^t_{{\cal E}_j}\,B)({\cal E}_i, E_a) -(\nabla^t_{E_a}\,B)({\cal E}_i, {\cal E}_j)\\
 && =\,\nabla^t_{{\cal E}_i}\,B({\cal E}_j, E_a)
 -B(\nabla^t_{{\cal E}_i}\,E_a, {\cal E}_j)
 +\nabla^t_{{\cal E}_j}\,B({\cal E}_i, E_a)
 -B(\nabla^t_{{\cal E}_j}\,E_a, {\cal E}_i) \\
 &&
 -\,\nabla^t_{E_a}\,B({\cal E}_i, {\cal E}_j)
 + B(\nabla^t_{E_a}\,{\cal E}_i, {\cal E}_j)
 + B(\nabla^t_{E_a}\,{\cal E}_j, {\cal E}_i).
\end{eqnarray*}
Using Lemma~\ref{prop-Ei-a}, we rewrite \eqref{E-10terms} as
\begin{eqnarray}\label{E-7terms}
\nonumber
 && \dt\<\tilde h,\tilde h\> =
 -\sum\nolimits_{\,i,j,a}\epsilon_{i}\epsilon_{j}\epsilon_{a}\nabla^t_{E_a}\,B({\cal E}_i,
 {\cal E}_j)\,g({\tilde h}({\cal E}_i, {\cal E}_j), E_a)\\
\nonumber
  && -\sum\nolimits_{\,i,j} \epsilon_{i}\epsilon_{j}
  \big[\,g({\tilde h}(B^{\sharp} {\cal E}_i, {\cal E}_j), {\tilde h}({\cal E}_i, {\cal E}_j))
  +g({\tilde h}(B^{\sharp} {\cal E}_j, {\cal E}_i), {\tilde h}({\cal E}_i, {\cal E}_j))\big] \\
\nonumber
 && -\sum\nolimits_{\,i,j,a} \epsilon_{i}\epsilon_{j}\epsilon_{a}\,g(\nabla^t_{{\cal E}_i}((B^{\sharp}{\cal E}_j)^{\top})
 +\nabla^t_{{\cal E}_j} ((B^{\sharp}{\cal E}_i)^{\top}), E_a)\,g({\tilde h}({\cal E}_i, {\cal E}_j), E_a) \\
\nonumber
 && -\sum\nolimits_{\,i,j,a} \epsilon_{i}\epsilon_{j}\epsilon_{a} \,g(\nabla^t_{(B^{\sharp}{\cal E}_j)^{\top}}{\cal E}_i
 +\nabla^t_{(B^{\sharp}{\cal E}_i)^{\top}}{\cal E}_j, E_a)\,g({\tilde h}({\cal E}_i, {\cal E}_j), E_a) \\
\nonumber
 && +\sum\nolimits_{\,i,j,a} \epsilon_{i}\epsilon_{j}\epsilon_{a} \, (\nabla^t_{{\cal E}_i} B({\cal E}_j, E_a)
 +\nabla^t_{{\cal E}_j} B({\cal E}_i, E_a))\,g({\tilde h}({\cal E}_i, {\cal E}_j), E_a) \\
\nonumber
 && -\sum\nolimits_{\,i,j,a} \epsilon_{i}\epsilon_{j}\epsilon_{a} (B(\nabla^t_{{\cal E}_i} E_a, {\cal E}_j)
 + B(\nabla^t_{{\cal E}_j} E_a, {\cal E}_i))\,g({\tilde h}({\cal E}_i, {\cal E}_j), E_a) \\
 && +\sum\nolimits_{\,i,j,a} \epsilon_{i}\epsilon_{j}\epsilon_{a}(B(\nabla^t_{E_a}{\cal E}_i, {\cal E}_j)
 + B(\nabla^t_{E_a} {\cal E}_j, {\cal E}_i))\,g({\tilde h}({\cal E}_i, {\cal E}_j), E_a) .
\end{eqnarray}
From the definition $2\,{\rm Sym}(C)=C+C^*$ we have
\[
 \<2\sum\nolimits_{\,a} ({\tilde T}^{\sharp}_a {\tilde A}_a)^\flat, \, B\>
 =2\,\<{\rm Sym}(\sum\nolimits_{\,a} {\tilde T}^{\sharp}_a {\tilde A}_a)^\flat,\, B\>
 =\<\sum\nolimits_{\,a} [{\tilde T}^{\sharp}_a,\, {\tilde A}_a]^\flat, \, B\>
 = \<\widetilde{\cal K}^\flat,\,B\>.
\]
and we obtain \eqref{E-tildeh-gen}, using the following computations for all 7 
lines of \eqref{E-7terms}:
\begin{eqnarray*}
 && \sum\nolimits_{\,i,j} g({\tilde h}(B^{\sharp}{\cal E}_i, {\cal E}_j), {\tilde h}({\cal E}_i, {\cal E}_j))
 = \<\tilde{\cal A}^\flat,\ B\>,\\
 && \sum\nolimits_{\,i,j,a}\epsilon_i \epsilon_j\epsilon_a\,g(\nabla^t_{{\cal E}_i} (B^{\sharp}({\cal E}_j)^{\top}), E_a)
 \,g({\tilde h}({\cal E}_i, {\cal E}_j), E_a) \\
 &&\hskip10mm = \Div \<{\tilde\alpha}, B\> - \< ( \Div{\tilde\alpha} )_{ | \widetilde{\cal D} \times {\cal D}} + ( \Div \tilde \alpha )_{ |  {\cal D} \times \widetilde{\cal D} } ,\,B\>,\\
 && \sum\nolimits_{\,i,j,a}\epsilon_i\epsilon_j\epsilon_a\,g(\nabla^t_{({B^{\sharp} {\cal E}_j })^{\top}}{\cal E}_i, E_a)\,
 g({\tilde h}({\cal E}_i, {\cal E}_j), E_a) \\
 &&\hskip10mm = -\sum\nolimits_{\,i,a} \epsilon_i\,\epsilon_a\,B({\tilde A}_a({\cal E}_i),\ A_{i}(E_a) - T^{\sharp}_{i}(E_a))
 = -\<\Lambda_{{\tilde\alpha}, \alpha - \theta}, B\>,\\
 && \sum\nolimits_{\,i,j,a} \epsilon_{i}\epsilon_{j}\epsilon_{a}\,g(\tilde h({\cal E}_i, {\cal E}_j), E_a)
 \nabla^t_{{\cal E}_i} B(E_a, {\cal E}_j) = \Div\<\tilde\alpha,\ B\> - \< ( \Div\tilde\alpha )_{ | \widetilde{\cal D} \times {\cal D}} + ( \Div \tilde \alpha )_{ |  {\cal D} \times \widetilde{\cal D} } ,\ B\>,\\
 && \sum\nolimits_{\,i,j,a} \epsilon_i\epsilon_j\epsilon_a\,g({\tilde h}({\cal E}_i, {\cal E}_j), E_a) \nabla^t_{E_a} B({\cal E}_i, {\cal E}_j)
 = \Div \<{\tilde h}, B\> - \<\Div {\tilde h},\,B\>,
 \end{eqnarray*}
\begin{eqnarray*}
 && \sum\nolimits_{\,i,j,a}\epsilon_i\epsilon_j\epsilon_a\,B(\nabla^t_{{\cal E}_i}\, E_a, {\cal E}_j)\,
 g({\tilde h}({\cal E}_i, {\cal E}_j), E_a) = -\<\tilde{\cal A}^\flat +\widetilde{\cal K}^\flat/2,\ B\>,\\
 &&\sum\nolimits_{\,i,j,a} \epsilon_i\epsilon_j\epsilon_a\,g({\tilde h}({\cal E}_i, {\cal E}_j), E_a)\,
 B(\nabla^t_{E_a} {\cal E}_i, {\cal E}_j) \\
 &&\hskip10mm =-\sum\nolimits_{\,i,a}\epsilon_i\epsilon_a\,B({\tilde A}_a({\cal E}_i), A_{i}(E_a) + T^{\sharp}_{i}(E_a))
 = - \<\Lambda_{ {\tilde\alpha}, \alpha + \theta }, B\>.
\end{eqnarray*}
As an example, we give a detailed computation of the 4th line above:
\begin{eqnarray*}
&& \sum\nolimits_{\,i,j,a} \epsilon_{i}\epsilon_{j}\epsilon_{a}\,g(\tilde h({\cal E}_i, {\cal E}_j), E_a)
 \nabla^t_{{\cal E}_i} B(E_a, {\cal E}_j) \\
  \eq \sum\nolimits_{\,i,j,a} \epsilon_{i}\epsilon_{j}\epsilon_{a}\, \nabla_i ( g( {\tilde h}( {\cal E}_i , {\cal E}_j ) , E_a ) B(E_a , {\cal E}_j ) ) - \sum\nolimits_{\,i,j,a} \epsilon_{i}\epsilon_{j}\epsilon_{a}\, B(E_a , {\cal E}_j ) \nabla_i ( g( {\tilde h}( {\cal E}_i , {\cal E}_j ) , E_a ) ) \\
 \eq \sum\nolimits_{\,i,j,a} \epsilon_{i}\epsilon_{j}\epsilon_{a}\, \nabla_i ( g( B(E_a , {\cal E}_j ) {\tilde A}_a {\cal E}_j , {\cal E}_i )  )
 - \sum\nolimits_{\,i,j,a} \epsilon_{i}\epsilon_{j}\epsilon_{a}\, B(E_a , {\cal E}_j ) ( \nabla_i g( {\tilde A}_a {\cal E}_j , {\cal E}_i ) ) \\
 \eq \sum\nolimits_{\,i} \epsilon_{i}\, g( \nabla_i ( \sum\nolimits_{\,j,a} \epsilon_{j}\epsilon_{a}  B(E_a , {\cal E}_j ) {\tilde A}_a {\cal E}_j ) , {\cal E}_i )
 - \sum\nolimits_{\,j,a} \epsilon_{j}\epsilon_{a}\, B(E_a , {\cal E}_j ) g( \sum\nolimits_{\,i} \epsilon_{i}\, \nabla_i {\tilde A}_a {\cal E}_j , {\cal E}_i ) ) \\
 \eq \Div^\perp \< B , {\tilde \alpha} \> - \< B_{ | \widetilde{\cal D} \times {\cal D}} + B_{ |  {\cal D} \times \widetilde{\cal D} }  , \Div^\perp {\tilde \alpha} \> \\
 \eq \Div\<\tilde\alpha,\ B\> - \< ( \Div\tilde\alpha )_{ | \widetilde{\cal D} \times {\cal D}} + ( \Div \tilde \alpha )_{ |  {\cal D} \times \widetilde{\cal D} } ,\ B\> .
\end{eqnarray*}
Note that while ${\tilde \alpha} = {\tilde \alpha}_{ | \widetilde{\cal D} \times {\cal D}} + { \tilde \alpha }_{ |  {\cal D} \times \widetilde{\cal D} }$, its divergence $\Div {\tilde \alpha}$ may not vanish on ${\cal D} \times {\cal D}$ or $\widetilde{\cal D} \times \widetilde{\cal D}$.

\textbf{Proof of} \eqref{E-tildeH-gen}. We compute for any $X\in T_x M$, using Lemmas~\ref{prop-Ei-a} and \ref{projections},
\begin{eqnarray}\label{Eq-tildeH-1}
 && g(\dt\tilde H, X)=\sum\nolimits_i\epsilon_i\,g(\dt((\nabla^t_{{\cal E}_i}\,{\cal E}_i)^\top), X)
 =\sum\nolimits_i\epsilon_i\,g(\dt(\nabla^t_{{\cal E}_i}\,{\cal E}_i), X^\top)\\
\nonumber
 &&\hskip-8pt =\!\sum\nolimits_i\!\epsilon_i\,g(
 \nabla_{(-B^{\sharp} {\cal E}_i)^{\top}}\,{\cal E}_i
 -\nabla_{(\frac12 B^{\sharp} {\cal E}_i)^{\bot}}\,{\cal E}_i
 -\nabla_{{\cal E}_i}((B^{\sharp}{\cal E}_i)^{\top})
 -\nabla_{{\cal E}_i}((\frac12 B^{\sharp} {\cal E}_i)^{\bot})
 +(\dt\nabla^t)_{{\cal E}_i}\,{\cal E}_i, X^\top).
\end{eqnarray}
Using known formula \eqref{eq2G} for $t$-derivative of the Levi-Civita connection, see \cite{topp},
we present \eqref{Eq-tildeH-1}$_5$ (i.e. the 5-th term in $g(\cdot,  X^\top)$ of \eqref{Eq-tildeH-1})
as the sum of 4 terms (we omit summation by $i$ below)
\begin{eqnarray}\label{Eq-tildeH-2}
\nonumber
  g((\dt\nabla^t)_{{\cal E}_i}\,{\cal E}_i,\,X^\top) \eq \nabla_{{\cal E}_i}(g(B^{\sharp}{\cal E}_i, X^\top))
  -g(B^{\sharp}{\cal E}_i, \nabla_{{\cal E}_i}\,X^\top) \\
  \minus \frac12\,\nabla_{X^\top}(B({\cal E}_i,{\cal E}_i)) -B((A_i+T^\sharp_i)(X^\top), {\cal E}_i).
\end{eqnarray}
We present the term \eqref{Eq-tildeH-1}$_3$ 
as the sum of 2 terms 
\begin{equation} \label{3a3b}
 -g(\nabla_{{\cal E}_i}((B^{\sharp}{\cal E}_i)^{\top}), X^\top)
 =-g(\nabla_{{\cal E}_i}(B^{\sharp}{\cal E}_i), X^\top)
 +g(\nabla_{{\cal E}_i}((B^{\sharp}{\cal E}_i)^{\bot}), X^\top),
\end{equation}
and then rewrite the term \eqref{3a3b}$_2$ as 
\begin{eqnarray*}
 g(\nabla_{{\cal E}_i}((B^{\sharp}{\cal E}_i)^{\bot}), X^\top)
 \eq -g((B^{\sharp}{\cal E}_i)^{\bot}, \nabla_{{\cal E}_i}\,X^\top)
 =-\sum\nolimits_j\epsilon_j\,g((B^{\sharp}{\cal E}_i)^{\bot}, {\cal E}_j)\,g({\cal E}_j, \nabla_{{\cal E}_i}\,X^\top)\\
 \eq \sum\nolimits_j\epsilon_j\,B({\cal E}_i, {\cal E}_j)\,g(\nabla_{{\cal E}_i}\,{\cal E}_j, X^\top)
 =\sum\nolimits_j\epsilon_j\,B({\cal E}_i, {\cal E}_j)\,g(\tilde h({\cal E}_i, {\cal E}_j), X^\top).
\end{eqnarray*}
Note that \eqref{Eq-tildeH-2}$_1$ + \eqref{Eq-tildeH-2}$_2$ + \eqref{3a3b}$_1$ =0.
For the sum \eqref{Eq-tildeH-1}$_2$ + \eqref{Eq-tildeH-1}$_4$ 
we get
\[
 g(-\nabla_{(\frac12\,B^{\sharp} {\cal E}_i)^{\bot}}\,{\cal E}_i -\nabla_{{\cal E}_i}((\frac12\,B^{\sharp} {\cal E}_i)^{\bot}), X^\top)
 = -g(\tilde h({\cal E}_i, B^{\sharp}{\cal E}_i), X^\top).
\]
For the term 
\eqref{Eq-tildeH-1}$_1$ we get
\begin{eqnarray*}
 -\sum\nolimits_{\,i}\epsilon_i\,g(\nabla_{{\cal E}_i}(g(B^{\sharp}{\cal E}_i, X^\top)), X^\top)
 =-\sum\nolimits_{\,i,a}\epsilon_a\,g(\nabla_{E_a}\,{\cal E}_i, X^\top)\,g(B^{\sharp}{\cal E}_i, E_a)\\
 =-\sum\nolimits_{\,i,a}\epsilon_a\,g((A_i+T^\sharp_i)(E_{a}, X^\top)\,B({\cal E}_i, E_a)
 =\<\,\<\alpha+\theta, X^\top\>,\,B\>.
\end{eqnarray*}
For the term 
\eqref{Eq-tildeH-2}$_3$ we get
\[
 -
 \sum\nolimits_i\epsilon_i\,\nabla_{X^\top}(B({\cal E}_i,{\cal E}_i)) = -
 X^\top(\tr_{\cal D}B).
\]
For the term 
\eqref{Eq-tildeH-2}$_4$ we get
\begin{eqnarray*}
 &&-\sum\nolimits_i\epsilon_i\,B((A_i+T^\sharp_i)(X^\top), {\cal E}_i)
 =-\sum\nolimits_i\epsilon_i\,g((A_i+T^\sharp_i)(X^\top), B^\sharp({\cal E}_i))\\
 && = -\sum\nolimits_{i,a}\epsilon_i\epsilon_a\,g(B^\sharp({\cal E}_i), E_a)\,g((A_i+T^\sharp_i)(X^\top), E_a)\\
 && = -\sum\nolimits_{i,a}\epsilon_i\epsilon_a\,B({\cal E}_i, E_a)\,g((A_i-T^\sharp_i)(E_a), X^\top)
 =\<\,\<\alpha-\theta, X^\top\>,\,B\>.
\end{eqnarray*}
Finally we collect results: 
\eqref{Eq-tildeH-1} = \eqref{Eq-tildeH-1}$_1$ +\eqref{Eq-tildeH-1}$_2$ + \eqref{3a3b}$_1$ + \eqref{3a3b}$_2$ + \eqref{Eq-tildeH-1}$_4$  + \eqref{Eq-tildeH-2}$_1$ + \eqref{Eq-tildeH-2}$_2$ + \eqref{Eq-tildeH-2}$_3$ + \eqref{Eq-tildeH-2}$_4$
to obtain
\[
 g(\dt\tilde H, X)= \<\,2\<\theta, X^\top\>,\,B\> -\frac12\,X^\top(\tr_{\cal D}B).
\]
Let $X=\tilde H$. Using $B(\tilde H, \tilde H)=0$
and $\tilde H(\tr_{\cal D}B^\sharp)=\Div((\tr_{\cal D}B^\sharp)\tilde H) -(\tr_{\cal D}B^\sharp)\Div\tilde H$, we get
\[
 \dt g(\tilde H, \tilde H) = 2\,g(\dt\tilde H, \tilde H)
 =\<\,4\,\<\theta, \tilde H\>,\,B\> -\Div((\tr_{\cal D}B^\sharp)\tilde H) +(\tr_{\cal D}B)\Div\tilde H.
\]
Finally note that $\tr_{\cal D}B=\<g,B\>$, and that completes the proof of \eqref{E-tildeH-gen}.

\smallskip

The computations for $h$ and $H$ are easier, since $B(X,Y)=0$ for $X,Y\in \widetilde{\cal D}$.
We assume that $\nabla_Z {\cal E}_i\in\widetilde{\cal D}$ at a point $x\in M$ for all $Z \in TM$.

\noindent\textbf{Proof of} \eqref{E-h-gen}. We observe that
\begin{equation*}
 \dt\<h,h\> = \sum\nolimits_{a,b}\epsilon_a\epsilon_b\,\big(B(h(E_a, E_b), h(E_a, E_b)) +2\,g(\dt h(E_a, E_b), h(E_a, E_b))\big),
\end{equation*}
where, using Lemma~\ref{projections} and formula \eqref{eq2G} for $\dt\nabla^t$, we compute
\begin{eqnarray*}
 &&  g(\dt h(E_a, E_b), h(E_a, E_b)) =
 \frac{1}{2}\,g(\dt ((\nabla^t_{E_a}\,E_b +\nabla^t_{E_b}\,E_a)^{\perp}), h(E_a, E_b))\\
 && = \frac{1}{2}\,g(\dt(\nabla^t_{E_a}\,E_b +\nabla^t_{E_b}\,E_a), h(E_a, E_b))
 = \frac{1}{2}\,g((\dt\nabla^t)_{E_a} E_b +(\dt\nabla^t)_{E_b} E_a, h(E_a, E_b)) \\
 && =
 \frac{1}{2}\sum\nolimits_i\epsilon_i\,g(h(E_a, E_b), {\cal E}_i)
 \big(\nabla^t_{E_a} B({\cal E}_i, E_b)+ \nabla^t_{E_b} B({\cal E}_i, E_a) \\
 &&\quad -\,2\,B( h(E_a, E_b), {\cal E}_i) + B(\nabla^t_{{\cal E}_i}\,E_a, E_b) + B(\nabla^t_{{\cal E}_i}\,E_b, E_a)\big).
\end{eqnarray*}
We used in the above
\begin{eqnarray*}
 && 2\, g((\dt \nabla^t)_{E_a} E_b, {\cal E}_i) = \nabla^t_{E_a} B({\cal E}_i, E_b)+ \nabla^t_{E_b} B({\cal E}_i, E_a) \\
 &&-2 B( h(E_a, E_b), {\cal E}_i) + B(\nabla^t_{{\cal E}_i}\,E_a, E_b) + B(\nabla^t_{{\cal E}_i}\,E_b, E_a).
\end{eqnarray*}
Note that
\begin{eqnarray*}
 && \sum\nolimits_{a,b,i}\epsilon_{a}\epsilon_{b}\epsilon_{i}\, g(h (E_a, E_b), {\cal E}_i) \nabla^t_{E_a}  B(E_b, {\cal E}_i)
 = \Div \<B, \alpha\> - \<\ ( \Div \alpha )_{ | \widetilde{\cal D} \times {\cal D}} + ( \Div \alpha )_{ |  {\cal D} \times \widetilde{\cal D} } ,\ B\>,\\
 && \sum\nolimits_{a,b,i} \epsilon_{a} \epsilon_{b} \epsilon_{i}\, g(h (E_a, E_b), {\cal E}_i) B(\nabla^t_{{\cal E}_i} E_a, E_b)
 = -\<\Lambda_{ \alpha, {\tilde\alpha} + { \tilde \theta } }, B\>.
\end{eqnarray*}
Finally, we obtain \eqref{E-h-gen}:
\begin{eqnarray*}
 \dt \<h,h\> \eq \sum\nolimits_{a,b} \epsilon_a \epsilon_b B(h(E_a, E_b),h(E_a, E_b))
 + 2\Div \<B, \alpha\> - 2\< ( \Div \alpha )_{ | \widetilde{\cal D} \times {\cal D}} + ( \Div \alpha )_{ | {\cal D} \times \widetilde{\cal D} } ,\ B\> \\
 && - 2 \sum\nolimits_{a,b} \epsilon_a \epsilon_b B(h(E_a, E_b),h(E_a, E_b))
 - 2\,  \<\Lambda_{ \alpha, {\tilde\alpha} + { \tilde \theta } },\ B\>.
\end{eqnarray*}
\textbf{Proof of} \eqref{E-H-gen}. We observe that
\begin{equation*}
 \dt  g(H,H) = B(H,H) + 2\,g(\dt H, H),
\end{equation*}
where, using Lemma~\ref{projections} and formula \eqref{eq2G} for $\dt\nabla^t$, we obtain
\begin{eqnarray*}
 && g(\dt H, H) = \sum\nolimits_a\epsilon_a g(\dt((\nabla^t_{E_a} E_a)^{\perp}), H)
 =\sum\nolimits_a\epsilon_a g((\dt\nabla^t)_{E_a} E_a, H) \\
 && = - B(H,H) +\sum\nolimits_a\epsilon_a B(\nabla^t_H E_a, E_a)
 +\sum\nolimits_{a,i}\epsilon_a\epsilon_i\,g(H, {\cal E}_i)\,\nabla^t_{E_a} B(E_a, {\cal E}_i).
\end{eqnarray*}
We have
\begin{eqnarray*}
 \sum\nolimits_a\epsilon_a\,B(\nabla^t_H E_a, E_a) \eq \<\,\<{\tilde\theta} - {\tilde\alpha}, H\>,\ B\>,\\
 \sum\nolimits_{a,i}\epsilon_a\epsilon_i\,g(H, {\cal E}_i)\,\nabla^t_{E_a} B(E_a, {\cal E}_i)
 \eq \Div(B^{\sharp} H)^{\top} + B(H, {\tilde H}) -\<{\tilde \delta}_H,\ B\>.
\end{eqnarray*}
Hence,
\begin{equation*}
 \dt g(H,H) = -B(H,H) + 2\,\big(\<\,\< \tilde\theta - {\tilde\alpha}, H\>\,, B\> + \Div (B^{\sharp} H)^{\top}
 + B(H, {\tilde H}) -\<{\tilde \delta}_H,\ B\>\big).
\end{equation*}
Finally, using $B(H, {\tilde H}) = \<{\rm Sym}(H^{\flat} \otimes {\tilde H}^{\flat}),\ B\>$, we obtain \eqref{E-H-gen}.

\noindent\textbf{Proof of} \eqref{E-tildeT-gen}. We compute
\begin{eqnarray*}
 \dt\<\tilde T, \tilde T\> \eq
 \sum\nolimits_{i,j}\epsilon_i\epsilon_j\,\dt g({\tilde T}({\cal E}_i, {\cal E}_j), {\tilde T} ({\cal E}_i, {\cal E}_j)) \\
 \eq \sum\nolimits_{i,j}\epsilon_i\epsilon_j\,\big(
 \dt B({\tilde T}({\cal E}_i, {\cal E}_j), {\tilde T} ({\cal E}_i, {\cal E}_j))
 +2\,g(\dt{\tilde T}({\cal E}_i, {\cal E}_j), {\tilde T} ({\cal E}_i, {\cal E}_j))
 \big).
\end{eqnarray*}
For the last term of the above, by symmetry $(\dt\nabla)_{{\cal E}_i}\,{\cal E}_j=(\dt\nabla)_{{\cal E}_j}\,{\cal E}_i$ and omitting sum,
we get
\begin{eqnarray} \label{tildeTcomp}
 && 2\,g(\dt{\tilde T}({\cal E}_i, {\cal E}_j), {\tilde T} ({\cal E}_i, {\cal E}_j)) =
 g(\dt((\nabla^t_{{\cal E}_i}\,{\cal E}_j - \nabla^t_{{\cal E}_j}\,{\cal E}_i)^{\top}), {\tilde T}({\cal E}_i, {\cal E}_j)) \nonumber \\
 && =g(\dt(\nabla^t_{{\cal E}_i}\,{\cal E}_j - \nabla^t_{{\cal E}_j}\,{\cal E}_i), {\tilde T}({\cal E}_i, {\cal E}_j))
 +g(B^\sharp((\nabla^t_{{\cal E}_i}\,{\cal E}_j - \nabla^t_{{\cal E}_j}\,{\cal E}_i)^{\bot}), {\tilde T}({\cal E}_i, {\cal E}_j)) \nonumber \\
 &&= g( \nabla_{(\dt{\cal E}_i)^\top}\,{\cal E}_j +\nabla_{(\dt{\cal E}_i)^\bot}\,{\cal E}_j
 +\nabla_{{\cal E}_i}\,((\dt{\cal E}_j)^\top) +\nabla_{{\cal E}_i}\,((\dt{\cal E}_j)^\bot) \nonumber \\
 && -\nabla_{(\dt{\cal E}_j)^\top}\,{\cal E}_i -\nabla_{(\dt{\cal E}_j)^\bot}\,{\cal E}_i
 -\nabla_{{\cal E}_j}\,((\dt{\cal E}_i)^\top) -\nabla_{{\cal E}_j}\,((\dt{\cal E}_i)^\bot)
 ,\,{\tilde T} ({\cal E}_i, {\cal E}_j)).
\end{eqnarray}
We will compute 8 terms in \eqref{tildeTcomp}
separately.
First we calculate
\begin{eqnarray*}
 && g(\tilde T((\dt{\cal E}_i)^\bot,{\cal E}_j),X^\top)
 =-\frac12\,g(\tilde T((B^\sharp{\cal E}_i)^\bot,{\cal E}_j),X^\top)\\
 && = -\frac12\sum\nolimits_{a}\epsilon_a\,
 g(\tilde T((B^\sharp{\cal E}_i)^\bot,E_a)\,g(E_a,X^\top)
 =-\frac12\sum\nolimits_{a}\epsilon_a\,
 g(\tilde T^\sharp_a(B^\sharp{\cal E}_i),{\cal E}_j) \,g(E_a,X^\top).
\end{eqnarray*}
Then, assuming $X=\tilde T({\cal E}_i,{\cal E}_j)$,
we find the sum \eqref{tildeTcomp}$_2$ + \eqref{tildeTcomp}$_8$: 
\begin{eqnarray*}
 &&2\sum\nolimits_{i,j,a}\epsilon_i\epsilon_j\epsilon_a\,
 (-\frac12)\,g(\tilde T^\sharp_a(B^\sharp{\cal E}_i),{\cal E}_j)
 \,g(E_a,\tilde T({\cal E}_i,{\cal E}_j))\\
 && =\sum\nolimits_{i,j,a}\epsilon_i\epsilon_j\epsilon_a\,
 g(\tilde T^\sharp_a{\cal E}_j,B^\sharp{\cal E}_i)\,
 g(\tilde T^\sharp_a{\cal E}_i,B^\sharp{\cal E}_j) \\
 &&=\sum\nolimits_{i,j,a,k}\epsilon_i\epsilon_j\epsilon_a\epsilon_k\,
 g(\tilde T^\sharp_a{\cal E}_j,{\cal E}_k)\,B({\cal E}_j,{\cal E}_k)\,g(\tilde T^\sharp_a{\cal E}_i,{\cal E}_j)\\
 &&=-\sum\nolimits_{i,a,k}\epsilon_i\epsilon_a\epsilon_k\,
 g(\tilde T^\sharp_a{\cal E}_k,\tilde T^\sharp_a{\cal E}_i)
 \,B({\cal E}_j,{\cal E}_k)
 =\sum\nolimits_{i,a,k}\epsilon_i\epsilon_a\epsilon_k\,
 g((\tilde T^\sharp_a)^2{\cal E}_k,{\cal E}_i)\,B({\cal E}_i,{\cal E}_k).
\end{eqnarray*}
For \eqref{tildeTcomp}$_4$ + \eqref{tildeTcomp}$_6$ 
we have the same, thus \eqref{tildeTcomp}$_2$ +\eqref{tildeTcomp}$_4$ +\eqref{tildeTcomp}$_6$ +\eqref{tildeTcomp}$_8$ = $2\,\<\tilde{\cal T},\, B\>$.
For \eqref{tildeTcomp}$_1$, which is equal to \eqref{tildeTcomp}$_5$,
we have
\begin{eqnarray*}
 &&-\sum\nolimits_{i,j}\epsilon_i\epsilon_j\,
 g(\nabla_{(B^\sharp{\cal E}_i)^\top} {\cal E}_j,\tilde T({\cal E}_i,{\cal E}_j))
 =-\sum\nolimits_{i,j,a}\epsilon_i\epsilon_j\epsilon_a\,
 g(\nabla_{E_a}{\cal E}_j, \tilde T({\cal E}_i,{\cal E}_j))
 B({\cal E}_i,E_a)\\
 &&=\sum\nolimits_{i,j,a}\epsilon_i\epsilon_j\epsilon_a\,
 g((A_j+T^\sharp_j)E_a,\tilde T({\cal E}_i,{\cal E}_j))\,B({\cal E}_i,E_a)\\
 &&=\sum\nolimits_{i,j,a,b}\epsilon_i\epsilon_j\epsilon_a\epsilon_b\,
 g((A_j+T^\sharp_j)E_a,E_b)\,
 g(\tilde T^\sharp_b{\cal E}_i,{\cal E}_j)\,B({\cal E}_i,E_a)\\
 &&=-\sum\nolimits_{i,j,b}\epsilon_i\epsilon_j\epsilon_b\,
 B({\cal E}_i,(A_j-T^\sharp_j)E_b)\,
 g(\tilde T^\sharp_b{\cal E}_j,{\cal E}_i)\\
 &&=-\sum\nolimits_{j,b}\epsilon_j\epsilon_b\,
 B(\tilde T^\sharp_b{\cal E}_j,(A_j-T^\sharp_j)E_b)
 = \<\Lambda_{\tilde\theta,\theta-\alpha},\,B\>.
\end{eqnarray*}
Thus, 
\eqref{tildeTcomp}$_1$ + \eqref{tildeTcomp}$_5$= $\<2\,\Lambda_{\tilde\theta,\theta-\alpha},\,B\>$.
For the term \eqref{tildeTcomp}$_3$
we have
\begin{eqnarray*}
 &&\hskip-5pt\sum\nolimits_{i,j}\epsilon_i\epsilon_j\,g(\nabla_{{\cal E}_i}((-B^\sharp{\cal E}_j)^\top),{\tilde T}({\cal E}_i, {\cal E}_j))
 =\sum\nolimits_{i,j,a}\epsilon_i\epsilon_j\epsilon_a\,
 g(\nabla_{{\cal E}_i}((-B^\sharp{\cal E}_j)^\top),E_a)\,g({\tilde T}({\cal E}_i, {\cal E}_j),E_a)\\
 && =-\sum\nolimits_{i,j,a}\epsilon_i\epsilon_j\epsilon_a\,\big(
 \nabla_{{\cal E}_i}\,g((B^\sharp{\cal E}_j)^\top,E_a)
 -g((B^\sharp{\cal E}_j)^\top,\nabla_{{\cal E}_i}\,E_a)\big)
 g({\tilde T}^\sharp_a({\cal E}_i), {\cal E}_j)\\
 && =-\sum\nolimits_{i,j,a}\epsilon_i\epsilon_j\epsilon_a\,\big(
  \nabla_{{\cal E}_i}\,(g(B({\cal E}_j,E_a)(-\tilde T^\sharp_a{\cal E}_j),{\cal E}_i))
  -B({\cal E}_j,E_a)\nabla_{{\cal E}_i}\,g(-\tilde T^\sharp_a{\cal E}_j,{\cal E}_i)\big) \\
 &&= \Div^\bot\<\tilde\theta,B\>
 -\sum\nolimits_{i,j,a}\epsilon_i\epsilon_j\epsilon_a\,
 B({\cal E}_j,E_a)\nabla_{{\cal E}_i}\,g(\tilde T^\sharp_a{\cal E}_j,{\cal E}_i)\\
 && = \Div^\bot\<\tilde\theta,B\> -\<\ ( \Div^\bot\tilde\theta )_{ | \widetilde{\cal D} \times {\cal D}} + ( \Div^\bot \tilde \theta )_{ |  {\cal D} \times \widetilde{\cal D} } ,B\> \\
&& = \Div\<\tilde\theta,B\> +\<\,\<\tilde\theta,B\>,H\>
 -\< ( \Div \tilde\theta )_{ | \widetilde{\cal D} \times {\cal D}} + ( \Div \tilde \theta )_{ |  {\cal D} \times \widetilde{\cal D} } ,B\> -\<\,\<\tilde\theta,B\>,H\>.
\end{eqnarray*}
Thus, 
\eqref{tildeTcomp}$_3$+ \eqref{tildeTcomp}$_7$ = $2\,\Div\<\tilde\theta,B\>-2\,\< ( \Div\tilde\theta )_{ | \widetilde{\cal D} \times {\cal D}} + ( \Div \tilde \theta )_{ |  {\cal D} \times \widetilde{\cal D} } ,\,B\>$.
Using the above, we obtain \eqref{E-tildeT-gen}.

\noindent\textbf{Proof of} \eqref{E-T-gen}. We calculate using Lemma~\ref{projections},
\begin{eqnarray*}
 && g(\dt T(E_a, E_b), X^\bot) = \frac{1}{2}\, g(\dt ((\nabla^t_{E_a} E_b - \nabla^t_{E_b} E_a)^\bot), X^{\perp})\\
 && = \frac{1}{2}\,g(\dt(\nabla^t_{E_a} E_b - \nabla^t_{E_b} E_a), X^{\perp})
 -\frac{1}{2}\, g(B^{\sharp} ((\nabla^t_{E_a} E_b - \nabla^t_{E_b} E_a )^{\perp}), (X^\perp)^{\top}) = 0.
\end{eqnarray*}
Then we obtain \eqref{E-T-gen}:
\begin{eqnarray*}
 \dt\<T,T\> \eq \dt\sum\nolimits_{a,b}\epsilon_a\epsilon_b\,g(T(E_a, E_b), T(E_a, E_b)) \\
 \eq
 \sum\nolimits_{a,b}\epsilon_a\epsilon_b\,\big(B(T(E_a, E_b), T(E_a, E_b)) + 2\,g(\dt T(E_a, E_b), T(E_a, E_b))\big)  \\
 \eq  \sum\nolimits_{a,b}\epsilon_a\epsilon_b\,B(T(E_a, E_b), T(E_a, E_b)) = -\<\Phi_T,\,B\>.
\end{eqnarray*}
This completes the proof.
\end{proof}

\begin{cor}\label{L-H2h2-D}
For  $g^\perp$-variations we have
\begin{subequations}
\begin{eqnarray}\label{E-h2T2-D1}
\nonumber
 \dt\widetilde{\Sm}_{\,\rm ex} \eq
 \<(\Div {\tilde H})\,g^{\perp} + 4 \<\theta,\ {\tilde H}\> - \Div {\tilde h}
 +4\,\Lambda_{{\tilde\alpha}, \theta }
 -\widetilde{\cal K}^\flat,\ B\> \nonumber \\
 && +\, \Div (\<{\tilde h}, B\> -(\tr_{\cal D} B) {\tilde H}),\\
\label{E-h2T2-D1b}
 \dt {\Sm}_{\,\rm ex} \eq \<-\Phi_{h} + 2 \<{\tilde\theta} - {\tilde\alpha},\ H\> + 2\,H^{\flat}\otimes {\tilde H}^{\flat}
 - 2\,{\tilde\delta}_{H} + 2 ( \Div \alpha )_{\widetilde{\cal D} \times {\cal D}} + 2 ( \Div \alpha )_{ |  {\cal D} \times \widetilde{\cal D} } \nonumber \\
 && +\, 2\, \Lambda_{\alpha, {\tilde\alpha} + {\tilde\theta}},\ B\>  +2\Div\big((B^{\sharp} H)^{\top} - \<\alpha, B\> \big) .
\end{eqnarray}
\end{subequations}
\end{cor}

\begin{proof} Formula \eqref{E-h2T2-D1} follows from  \eqref{E-tildeh-gen} and \eqref{E-tildeH-gen},
and \eqref{E-h2T2-D1b} follows from  \eqref{E-h-gen} and \eqref{E-H-gen}.
\end{proof}

Similarly as Lemma \ref{prop-Ei-a}, one can prove the following
\begin{lem} \label{framegenvar}
Let $\{E_a,\,{\cal E}_{i}\}$ be a local $(\widetilde{\cal D},\,{\cal D})$-adapted and $g$-orthonormal frame. For any variation $g_t$ the frame evolving according to equations:
\begin{eqnarray*}
&& E_a(0) = E_a , \quad \dt E_a = - \frac{1}{2} (B^\sharp E_a )^\top, \\
&& {\cal E}_i (0) = {\cal E}_i , \quad \dt {\cal E}_i = - (B^\sharp {\cal E}_i )^\top - \frac{1}{2} (B^\sharp {\cal E}_i )^\perp,
\end{eqnarray*}
where $B = \dt g_t$, remains an orthonormal frame adapted to $\widetilde{\cal D}$ and ${\cal D}(t)$.

For any ${g^\top}$-variation the frame evolving according to equations:
\begin{eqnarray*}
&& E_a(0) = E_a , \quad \dt E_a = - \frac{1}{2} (B^\sharp E_a )^\top, \\
&& {\cal E}_i (0) = {\cal E}_i , \quad \dt {\cal E}_i = 0,
\end{eqnarray*}
where $B = \dt g_t$, remains an orthonormal frame adapted to $\widetilde{\cal D}$ and $\cal D$.
\end{lem}

Lemma \ref{projections} remains true without any changes for both $g_t$- and ${g^\top}$-variations, and Proposition \ref{propvar1} has the following analogue.


\begin{prop} \label{propvar2} For  ${g^\top}$-variations we have
\begin{subequations}
\begin{eqnarray}\label{E-tildeh-gen2}
 &&\hskip-9mm
 \dt\<{\tilde h}, {\tilde h}\> = \< \Phi_{\tilde h} ,B \> - B({\tilde H},{\tilde H}),\\
\label{E-tildeH-gen2}
 &&\hskip-9mm \dt g({\tilde H}, {\tilde H}) = -B({\tilde H},{\tilde H}) 
 ,\quad\\
\label{E-h-gen2}
 &&\hskip-9mm \dt \<h,h\> = \<\Div{ h} + {\cal K}^\flat,\, B\> - \Div\< h, B\>, \\
\label{E-H-gen2}
 &&\hskip-9mm \dt g(H, H) =  \<\,(\Div{ H})\,{{g^\top}} \, B \> -\Div( (\tr_{\widetilde{\cal D}} B^\sharp) H ),\\
\label{E-tildeT-gen2}
 &&\hskip-9mm \dt\<{\tilde T}, {\tilde T}\> = - \<\Phi_{\tilde T}, B\> ,\\
\label{E-T-gen2}
 &&\hskip-9mm \dt\<T, T\> = 2\,\<\, {\cal T}^\flat ,\, B\,\>.
\end{eqnarray}
\end{subequations}
\end{prop}

\begin{proof}
One can prove the above analogue of Proposition \ref{propvar1} by direct computations, but careful comparison of Lemmas \ref{prop-Ei-a} and \ref{framegenvar} indicates that it is enough to take formulas dual (with respect to interchanging $\widetilde{\cal D}$ and ${\cal D}$) to \eqref{E-tildeh-gen} -- \eqref{E-T-gen}, and assume in them that $B = B_{ | \widetilde{\cal D} \times \widetilde{\cal D} }$.
\end{proof}

\begin{rem} \rm
Similarly, one can prove that variation formulas for general variations $g_t$ are sums of the corresponding formulas from Proposition \ref{propvar1} and Proposition \ref{propvar2}. This follows from the fact that every infinitesimal variation of $g$ can be decomposed into the sum of infinitesimal $g^\perp$- and ${g^\top}$-variations.
\end{rem}

\subsection{Euler-Lagrange equations}
\label{subsec:EU-critical}

In this section we present the Euler-Lagrange equations for the action \eqref{E-Jmix}.

We shall consider two different kinds of variations of metric. For arbitrary variations of the metric, the Euler-Lagrange equation is simply a condition for vanishing of the gradient of the functional: $\delta J_{\rm mix,\widetilde{\cal D},\Omega}(g)=0$, where
 \[
 \dt J_{\rm mix,\widetilde{\cal D},\Omega} = \int_{\Omega} \< \delta J_{\rm mix,\widetilde{\cal D},\Omega} , B \> {\rm d} \vol_g
 \]
for any variation $g_t$ with $B = \dt g_t$. In analogue to the Einstein-Hilbert action, one can also consider variations preserving the volume of $\Omega$. Then, the Euler-Lagrange equation has the following form: $\delta J_{\rm mix,\widetilde{\cal D},\Omega}(g)=\lambda g$, where $\lambda \in C^\infty(M)$ is an arbitrary function \cite{besse}. Unfortunately, in our case the functional $J_{\rm mix,\widetilde{\cal D},\Omega}$ is not a Riemannian functional (i.e. it is not invariant under all diffeomorphisms of $M$), hence $\lambda$ will not be, in general, a constant \cite{besse}.

There is however a natural way to obtain a non-homogeneous Euler-Lagrange equation with a more specific right-hand side. In order to do so, for a fixed domain of integration $\Omega$ in \eqref{E-Jmix}, we shall consider variations obtained from $g^\perp$-variations (and, separately, $g^\top$-variations) through rescaling the varying component of the metric by a function constant on $M$, depending only on the parameter of variation.

\begin{defn} \rm
Let an open, relatively compact set $\Omega \subset M$ be the domain of integration in \eqref{E-Jmix}. Let $g_t$ be a $g^\perp$-variation \eqref{e:Var} with the support contained in $\Omega$, and let ${\cal D}(t)$ denote the $g_t$-orthogonal complement of $\widetilde{\cal D}$. 
A \emph{${\bar g}^\perp$-variation} of the metric $g$ is a family of metrics given by:
\begin{eqnarray}\label{e:Var-Bar}
&& \bar g_t := \phi_t^\perp g_t |_{ {\cal D}(t) \times {\cal D}(t)} + g |_{\widetilde{\cal D} \times \widetilde{\cal D}} , \nonumber \\
&& \phi_t^\perp := \big({\rm Vol}({\Omega}, g_t)/{\rm Vol}({\Omega}, g)\big)^{-2/p} . 
\end{eqnarray}
\end{defn}

Similarly, we define a \emph{${\bar g}^\top$-variation} by taking an arbitrary $g^\top$-variation \eqref{e:Vartop}, $g_t$, and rescaling the ${\widetilde{\cal D} \times \widetilde{\cal D}}$-component of $g_t$ by the function $\phi_t^\top := \big({\rm Vol}({\Omega}, g_t)/{\rm Vol}({\Omega}, g)\big)^{-2/n}$.


\begin{rem} \label{remarkbargvariations} \rm
For both ${\bar g}^\perp$- and ${\bar g}^\top$-variations we have ${\rm Vol}({\Omega},\bar g_t) = {\rm Vol}({\Omega}, g)$ for all $t$ for which $g_t$ is defined, see \cite{rz-1}. Obviously, every $g^\perp$-variation, that preserves the volume of $\Omega$, is a ${\bar g}^\perp$-variation (with $\phi_t^\perp \equiv 1$). On the other hand, since $\dt {\bar g}_t^\perp = \phi_t^\perp \dt g_t + (\dt \phi_t^\perp )g_t$, we see that ${\bar g}^\perp$-variations may not vanish on the boundary of $\Omega$ (unlike all $g^\perp$-variations). Thus, metrics critical with respect to ${\bar g}^\perp$-variations are also critical with respect to volume preserving $g^\perp$-variations, but the reverse statement may not be true. Obviously, the same considerations apply to ${\bar g}^\top$-variations.
\end{rem}

For every $f \in L^1 ({\Omega}, \; {\rm d} \, {\rm vol}_g)$, denote
the mean value of $f$ on ${\Omega}$ with respect to ${\rm d}\,{\rm vol}_g$ by
\[
 f({\Omega}, g) = {\rm Vol}^{-1}({\Omega}, g) \int_{\Omega} f\,{\rm d}\,{\rm vol}_g.
\]

\begin{lem}[see \cite{rz-1}]\label{divH+tildeH}
\begin{equation*}
 {\rm\frac{d}{dt}}\int_{\Omega} \Div(H + \tilde H)\,{\rm d}\vol_{g_t}\!=\!\left\{
 \begin{array}{cc}
   0, & \mbox{\rm for $g^\perp$-variations}, \\
   \!\!\!\frac{1}{2}\Div\big(\frac{2-p}{p} H {-} {\tilde H}\big)(\Omega, g)\int_{\Omega}(\tr_g B)\,{\rm d}\vol_g, &
   \!\ \mbox{\rm for ${\bar g}^\perp$-variations}.
 \end{array}\right.
\end{equation*}
\end{lem}

Next result compares ${\bar g}^\perp$- and $g^\perp$-variations of the action \eqref{E-Jmix}.

\begin{prop}[see \cite{rz-1}]\label{P-1-5}
Let $g_t$ be a $g^\perp$-variation with $B = \dt g_t$ and let ${\bar g}_t$ be a variation obtained from $g_t$ by \eqref{e:Var-Bar}. Then:
\begin{eqnarray} \label{gtandbargt}
 {\rm\frac{d}{dt}}\,J_{\rm mix,\widetilde{\cal D},\Omega}(\bar{g}_t)_{\,|\,t=0}
 ={\rm\frac{d}{dt}}\,J_{\rm mix,\widetilde{\cal D},\Omega}({g}_t)_{\,|\,t=0}
 - \frac12\Sm^*_{\rm mix}(\Omega,g) \int_{\Omega} \big(\tr_{g} {B} \big)\,{\rm d}\vol_g \,,
\end{eqnarray}
where
\begin{equation}\label{E-S-star}
 \Sm^*_{\rm mix} = \Sm_{\rm mix} -\frac2p\,\big(\Sm_{\,\rm ex} +2\,\<\tilde T,\tilde T\> -\<T,T\> + \Div H\big).
\end{equation}
\end{prop}

For ${\bar g}^\top$-variations we obtain equation of the same form as \eqref{gtandbargt}, only with
\begin{equation} \label{E-S-star-tilde}
\widetilde{\Sm}^*_{\rm mix} = \Sm_{\rm mix} -\frac2n\,\big(\widetilde{\Sm}_{\,\rm ex} +2\,\< T, T\> -\<\tilde T, \tilde T\> + \Div {\tilde H} \big)
\end{equation}
in place of $\Sm^*_{\rm mix}$. Note that $\widetilde{\Sm}^*_{\rm mix}$ is dual to $\Sm^*_{\rm mix}$ with respect to interchanging $\widetilde{\cal D}$ and ${\cal D}$.


\begin{thm}[\bf Euler-Lagrange equations]\label{T-main00}
A metric $g\in{\rm Riem}(M,\widetilde{\cal D},{\cal D})$ is critical for the action~\eqref{E-Jmix}
with respect to ${\bar g}^\perp$-variations if and only if
\begin{eqnarray}\label{E-main-0i}
 \nonumber
 &&\hskip-7mm {r}_{\cal D} -\<\tilde h,\,\tilde H\> +\widetilde{\cal A}^\flat -\widetilde{\cal T}^\flat
 +\Phi_h +\Phi_T +\Psi -{\rm Def}_{\cal D}\,(H) +\widetilde{\cal K}^\flat \\
 && =\frac12\,\big(\Sm_{\rm mix} - \Sm^*_{\rm mix}(\Omega,g) + \Div(\tilde H - H)\big)\,g^\perp
 , \\
\label{E-main-0ii}
\nonumber
 &&\hskip-7mm 2 \<\theta,\, {\tilde H}\> + ( \Div(\alpha - \tilde \theta) )_{ | \widetilde{\cal D} \times {\cal D}} + ( \Div (\alpha - \tilde \theta) )_{ |  {\cal D} \times \widetilde{\cal D} } + \<{\tilde \theta} - {\tilde\alpha}, H\> \\
 && + {\rm Sym}(H^{\flat}\otimes{\tilde H}^{\flat}) -\,{\tilde \delta}_{H} + 2\Lambda_{{\tilde\alpha}, \theta} + \Lambda_{\alpha, {\tilde\alpha}} + \Lambda_{\theta, {\tilde\theta}} = 0,
\end{eqnarray}
where $\Sm^*_{\rm mix}$ is given by \eqref{E-S-star}.

A metric $g\in{\rm Riem}(M,\widetilde{\cal D},{\cal D})$ is critical for the action~\eqref{E-Jmix}
with respect to ${\bar g}^\top$-variations if and only if
\begin{eqnarray}
\label{E-main-0iii}
 \nonumber
 &&\hskip-7mm {r}_{\widetilde{\cal D}} -\<h,\,H\> +{\cal A}^\flat -{\cal T}^\flat
 +\Phi_{\tilde h} +\Phi_{\tilde T} +\widetilde\Psi -{\rm Def}_{\widetilde{\cal D}}\,(\tilde H) +{\cal K}^\flat \\
 && =\frac12\,\big(\widetilde\Sm_{\rm mix} - \widetilde\Sm^*_{\rm mix}(\Omega,g)
 + \Div(H -\tilde H)\big)\,{g^\top} ,
\end{eqnarray}
where $\widetilde{\Sm}^*_{\rm mix}$ is given by \eqref{E-S-star-tilde}.
\end{thm}

\begin{proof} 
Let $g_t$ be a $g^\perp$-variation.
 By \eqref{eq-ran-ex} and Lemma~\ref{divH+tildeH}, we have
\begin{equation*}
 {\rm\frac{d}{dt}}\,J_{\rm mix,\widetilde{\cal D},\Omega}(g_t)_{\,|\,t=0} = {\rm\frac{d}{dt}}\int_{\Omega} Q(g_t)\,{\rm d}\vol_{g_t}\,,
\end{equation*}
where $Q(g):=\Sm_{\rm mix} -\Div(H+\tilde H)$ can be presented using \eqref{eq-ran-ex} as
\begin{equation}\label{E-Q-def}
 Q(g) = \Sm_{\,\rm ex}(g) +\widetilde\Sm_{\,\rm ex}(g) +\<T,T\>_{g} +\<\tilde T,\tilde T\>_{g}\,.
\end{equation}
Applying Corollary~\ref{L-H2h2-D} to \eqref{E-Q-def}, using \eqref{E-divP} and removing integrals of divergences of vector fields compactly supported in $\Omega$, we get
\begin{eqnarray}\label{E-Sc-var1a}
 &&\hskip-8mm \int_{\Omega}\dt Q(g_t)_{\,|\,t=0}\,{\rm d}\vol_g
 = \int_{\Omega}\big\<4\Lambda_{{\tilde\alpha}, \theta} -\Div{\tilde h}
 -\widetilde{\cal K}^\flat -\Phi_{h} -\Phi_T +2\,\tilde{\cal T}^\flat
  +\,4\,\<\theta,\,{\tilde H}\> \nonumber \\
 && + (\Div {\tilde H})g^{\perp} + 2 ( \Div\alpha )_{ | \widetilde{\cal D} \times {\cal D}} + 2 ( \Div \alpha )_{ |  {\cal D} \times \widetilde{\cal D} } + 2\,\Lambda_{\alpha, {\tilde\alpha} + {\tilde\theta}} +2\,\<{\tilde \theta} -{\tilde\alpha}, H\> \nonumber \\
 && +\,2\,{\rm Sym}(H^{\flat}\otimes{\tilde H}^{\flat}) - 2\,{\tilde\delta}_{H} + 2\Lambda_{{\tilde\theta}, \theta - \alpha}
 -2 ( \Div {\tilde \theta} )_{ | \widetilde{\cal D} \times {\cal D}} - 2 ( \Div \tilde \theta )_{ |  {\cal D} \times \widetilde{\cal D} } ,\ B \big\> \,{\rm d}\vol_g,
\end{eqnarray}
where ${B}=\{\dt g_t\}_{\,|\,t=0}$.
Notice that $\tr_{g} {B}=\<g^\perp,\,B\>$. Then by \eqref{E-Sc-var1a}, we have
\begin{eqnarray}\label{E-varJh-init2}
 \nonumber
 && {\rm\frac{d}{dt}}\,J_{\rm mix,\widetilde{\cal D},\Omega}({g}_t)_{|\,t=0} = \int_{\Omega}
 \big\<4\,\Lambda_{{\tilde\alpha}, \theta} -\Div{\tilde h} -\widetilde{\cal K}^\flat
 - \Phi_{h} - \Phi_T +2\,\tilde{\cal T}^\flat  +\,4\,\<\theta,\,{\tilde H}\> \\
\nonumber
 && +2 ( \Div(\alpha -\tilde\theta) )_{ | \widetilde{\cal D} \times {\cal D}} + 2 ( \Div ( \alpha - \tilde \theta ) )_{ |  {\cal D} \times \widetilde{\cal D} }
 + 2 \Lambda_{\alpha, {\tilde\alpha} + {\tilde \theta}} + 2\<{\tilde \theta}
 - {\tilde\alpha}, H\> \\
 && +2\,{\rm Sym}(H^{\flat} \otimes {\tilde H}^{\flat}) -\,2\,{\tilde \delta}_{H} +2\,\Lambda_{{\tilde \theta}, \theta - \alpha}
  +\frac{1}{2}\,\big(\,{\Sm}_{\,\rm mix} + \Div({\tilde H} - H)\big)\, g^{\perp},\ B\big\>\,{\rm d}\vol_g.
\end{eqnarray}
By \eqref{E-varJh-init2} and Proposition~\ref{P-1-5}, we obtain
\begin{eqnarray}\label{E-Jmix-dt-fin}
\nonumber
 && {\rm\frac{d}{dt}}\,J_{\rm mix,\widetilde{\cal D},\Omega}(\bar g_t)_{\,|\,t=0} = \int_{\Omega}
 \<4\,\Lambda_{{\tilde\alpha}, \theta} -\Div {\tilde h} -\widetilde{\cal K}^\flat - \Phi_{h} - \Phi_T
 + 2\,\tilde{\cal T}^\flat +4 \<\theta,\,{\tilde H}\> \\
\nonumber
 && +\,2 ( \Div(\alpha - \tilde\theta) )_{ | \widetilde{\cal D} \times {\cal D}} + 2 ( \Div ( \alpha - \tilde \theta ) )_{ |  {\cal D} \times \widetilde{\cal D} }
 + 2\,\Lambda_{\alpha, {\tilde\alpha} + {\tilde \theta}} + 2\,\<{\tilde \theta} - {\tilde\alpha}, H\>
 +2\,{\rm Sym}(H^{\flat} \otimes {\tilde H}^{\flat}) \\
 && - 2\,{\tilde \delta}_{H} + 2\,\Lambda_{{\tilde \theta}, \theta - \alpha}
 +\frac{1}{2}\,\big(\,{\Sm}_{\,\rm mix} - {\Sm}_{\,\rm mix}^{*}(\Omega, g)
 +\Div({\tilde H} - H)\big)\,g^{\perp}, B\>\,{\rm d}\vol_g.
\end{eqnarray}
If $g$ is critical for $J_{\rm mix,\widetilde{\cal D},\Omega}$ with respect to ${\bar g}^\perp$-variations,
then the integral in \eqref{E-Jmix-dt-fin} is zero for arbitrary symmetric $(0,2)$-tensor $B$ vanishing on $\widetilde{\cal D} \times \widetilde{\cal D}$.
This yields the Euler-Lagrange equation, that we can decompose into two independent parts:
its $\cal D \times \cal D$ and $\widetilde{\cal D} \times {\cal D}$ components, obtaining the following
\begin{subequations}
\begin{eqnarray}\label{ElmixDD}
 && 
 \Div{\tilde h} +\widetilde{\cal K}^\flat + \Phi_{h} + \Phi_T -2\,\tilde{\cal T}^\flat
 =\frac{1}{2}\,\big({\Sm}_{\,\rm mix} - {\Sm}\,_{\rm mix}^{*}(\Omega, g) +\Div({\tilde H}- H)\big)\,g^{\perp}, \\
\label{ElmixDtildeD}
 && 
 2\,\<\theta,\, {\tilde H}\> + 2\Lambda_{{\tilde\alpha}, \theta }
 {+} ( \Div(\alpha - \tilde \theta) )_{ | \widetilde{\cal D} \times {\cal D}} + ( \Div  ( \alpha - \tilde \theta ) )_{ |  {\cal D} \times \widetilde{\cal D} }  \nonumber \\ &&
  {+} \Lambda_{\alpha, {\tilde\alpha} +{\tilde \theta}}
 {+}\<{\tilde \theta} - {\tilde\alpha}, H\> + {\rm Sym}(H^{\flat} \otimes {\tilde H}^{\flat})
 {-} {\tilde \delta}_{H} {+} \Lambda_{{\tilde \theta}, \theta - \alpha } =0.
\end{eqnarray}
\end{subequations}
Using
tensor $r_{{\cal D}}$ (Proposition~\ref{L-CC-riccati}) and replacing ${\Div}\,\tilde h$ in \eqref{ElmixDD}
according to \eqref{E-genRicN}, we rewrite \eqref{ElmixDD} as \eqref{E-main-0i}.
Using the properties $\Lambda_{P,Q} = \Lambda_{Q,P}$ and $\Lambda_{P, Q_{1} + Q_{2}} = \Lambda_{P,Q_1 }+ \Lambda_{P,Q_2 }$,
we rewrite \eqref{ElmixDtildeD} as~\eqref{E-main-0ii}.
Finally, using the fact that all variation formulas for $g^\top$-variations are dual to the ${\cal D} \times {\cal D}$ components of the variation formulas for $g^\perp$-variations, we can take the dual equation to \eqref{ElmixDD} to obtain the following Euler-Lagrange equation for ${\bar g}^\top$-variations:
\begin{equation} \label{ElmixDtildeDtilde}
 \Div{ h} +{\cal K}^\flat + \Phi_{\tilde h} + \Phi_{\tilde T} -2\,{\cal T}^\flat
 =\frac{1}{2}\,\big({\Sm}_{\,\rm mix} - \widetilde{\Sm}\,_{\rm mix}^{*}(\Omega, g) +\Div({ H}- {\tilde H})\big)\,g^{\top} .
\end{equation}
Using the dual of Proposition~\ref{L-CC-riccati} yields \eqref{E-main-0iii}.

Note that \eqref{E-main-0i} and \eqref{E-main-0iii} coincide with the equations obtained in~\cite{rz-1} for \textit{adapted variations} of metric, i.e. ${\bar g}^\perp$- and ${\bar g}^\top$-variations that were additionally required to preserve orthogonality of distributions $\widetilde{\cal D}$ and ${\cal D}$.
\end{proof}

Similarly to the discussion in Remark \ref{remarkbargvariations}, we can relate the Euler-Lagrange equations for different types of variations.

\begin{rem} \label{remarkAllVariations} \rm
To obtain the Euler-Lagrange equations for arbitrary, not necessarily preserving volume of $\Omega$, $g^\perp$-variations (resp. $g^\top$-variations), one should merely delete the mean value terms $\Sm^*_{\rm mix}$ from the Euler-Lagrange equations obtained for ${\bar g}^\perp$-variations (resp. ${\bar g}^\top$-variations).
To obtain the Euler-Lagrange equations for arbitrary variations $g_t$ 
preserving the volume of $\Omega$, one should replace both $\Sm^*_{\rm mix}$ and $\widetilde{\Sm}^*_{\rm mix}$ by the same, arbitrary function $\lambda \in C^\infty(M)$.
\end{rem}

A trivial example of metric critical for the action \eqref{E-Jmix} is the one of the metric product of manifolds, i.e., with both $\widetilde{\cal D}$ and ${\cal D}$ integrable and totally geodesic. In fact, it is difficult to find other critical points of \eqref{E-Jmix} for arbitrary variations of metric. However, there exist many interesting examples of metrics critical with respect to both ${\bar g}^\perp$- and ${\bar g}^\top$-variations, 
that will be presented in further sections. Despite the fact that ${\bar g}^\perp$- and ${\bar g}^\top$-variations are defined for a fixed domain of integration $\Omega$, we shall find metrics that are critical with respect to them regardless of the choice of $\Omega$. Also, note that ${\bar g}^\perp$- and ${\bar g}^\top$-variations generalize other variations considered in literature, e.g., variation among associated metrics on a contact manifold \cite{b2010}, discussed in Section \ref{SectionContact}.


\section{Particular cases}
\label{sec:main}

In this part of the paper we examine the Euler-Lagrange equations \eqref{E-main-0i}-\eqref{E-main-0iii}, 
assuming particu\-lar (co)dimension of the distribution $\widetilde{\cal D}$ or existence of an additional 
structure on the manifold $M$. In those special geometric settings, we obtain examples of metrics critical for the action \eqref{E-Jmix}, with respect to different variations previously discussed.

\subsection{Flows} \label{sectionFlows}

Let $\widetilde{\cal D}$ be spanned by a nonsingular vector field $N$, then it is tangent to the one-dimensional foliation by the flowlines of $N$.
 In~this case, $\Sm_{\rm mix}=\epsilon_N\Ric_{N}$, 
$R_N=R(N,\,\cdot\,)N$ is the Jacobi operator and the partial Ricci tensor takes a particularly simple form:
\[
 r_{\widetilde{\cal D}}=\epsilon_{N} \Ric_{N}\, g^\top ,\qquad
 r_{\,\cal D}=\epsilon_{N} (R_N)^\flat,
\]
We have $\tilde h = \tilde h_{sc} N$,
where $\tilde h_{sc} = \epsilon_{N} \<\tilde h, N\>$ is the scalar second fundamental form of~${\cal D}$. Let $\tilde A_N$ be the Weingarten operator associated to $\tilde h_{sc}$ and let $\tilde\tau_i=\tr\tilde A_N^{\,i}\ (i\ge0)$.
We have
$\Sm_{\rm ex} = g(H,H) - \< h,h \> = g(H,H) - g(H,H) = 0$, 
$\widetilde{\Sm}_{\,\rm ex}=\tilde\tau_1^2 -\tilde\tau_2$ and
\begin{eqnarray*}
 \Div N \eq \sum\nolimits_{\,i} \epsilon_{i}\,g(\nabla_{{\cal E}_i} N, {\cal E}_i)
 = -g(N,\sum\nolimits_{\,i}\! \epsilon_{i}\nabla_{{\cal E}_i}\,{\cal E}_i)
 = -g(N, \tilde H) = -\tilde\tau_1,\\
 \Div(\tilde\tau_1 N) \eq N(\tilde\tau_1) +\tilde\tau_1\Div N=N(\tilde\tau_1)-\tilde\tau_1^2.
\end{eqnarray*}
The curvature of the flow lines is $H=\epsilon_N\,\nabla_{N}\,N$.
 It is easy to see that \eqref{E-S-star} and \eqref{E-S-star-tilde} take the following form:
\begin{eqnarray} \label{Sstarflow}
&& \Sm^*_{\rm mix} = \epsilon_{N}\Ric_{N} -\,2(\frac2p\,\<\tilde T,\tilde T\> +\frac1p\,\Div H)
\\
\label{tildeSstarflow}
&& \widetilde{\Sm}^*_{\rm mix} = \epsilon_{N}\Ric_{N} -\,2(\epsilon_{N}(N(\tilde\tau_1)-\tilde\tau_2)-\<\tilde T,\tilde T\> ) .
\end{eqnarray}
From Theorem~\ref{T-main00} 
 we obtain the following.

\begin{cor}[\bf Euler-Lagrange equations]
\label{C-main02}
Let a distribution $\widetilde{\cal D}$ be spanned by a unit vector field $N$ on a manifold~$M$
with respect to $g\in{\rm Riem}(M,\widetilde{\cal D},{\cal D})$.
The metric $g$ is critical for the action \eqref{E-Jmix} with respect to ${\bar g}^\perp$-variations if and only if
\begin{eqnarray}\label{E-main-1i}
\nonumber
 && \epsilon_{N}\big(R_N +\tilde A_N^2 -(\tilde T^\sharp_N)^2 +[\tilde T_N^\sharp,\tilde A_N]\big)^\flat
 \!-\tilde\tau_1\tilde h_{sc} +H^\flat\otimes H^\flat -{\rm Def}_{\cal D}\,H  \\
 &&
 \hskip2mm
 =\frac12\,\big(\epsilon_{N} \Ric_{N} -\Sm^*_{\rm mix}(\Omega,g)
 +\Div(\epsilon_{N} \tilde\tau_1 N - H)\big)\,g^\perp , \\
 \label{E-main-3i}
 && {\Div}^{\perp} {\tilde T}^{\sharp}_{N} |_{\cal D} + 2\, ( {\tilde T}^{\sharp}_{N}(H))^\flat =0 ,  
\end{eqnarray}
where $\Sm^*_{\rm mix}$ is given by \eqref{Sstarflow}.

The metric $g$ is critical for the action \eqref{E-Jmix} with respect to ${\bar g}^\top$-variations if and only if
\begin{eqnarray}
\label{E-main-2i}
 && \epsilon_{N} \Ric_{N} + \widetilde{\Sm}^*_{\rm mix}(\Omega,g) -4 \<\tilde T, \tilde T\> -\Div(\epsilon_{N}\tilde\tau_1 N
 + H) =0 , 
\end{eqnarray}
where $\widetilde{\Sm}^*_{\rm mix}$ is given by \eqref{tildeSstarflow}.
\end{cor}

\proof
 An easy computation shows that
\begin{eqnarray}\label{E-Jmix-dim1}
\nonumber
 \widetilde{\cal A} \eq \epsilon_{N} \tilde A_N^2,\quad
  \<\tilde h_{sc} N,\,\tilde H\> = \tilde\tau_1\tilde h_{sc},\quad \Psi = H^\flat\otimes H^\flat,\quad
  \widetilde\Psi = (\epsilon_{N} \tilde\tau_2 - \<\tilde T, \tilde T\>)\, g^\top ,\\
\nonumber
 {\cal A} \eq g(H,H) \,\widetilde\id,\quad
 {\cal T}  = 0,\quad
 \<h,\,H\> = g(H,H) \, g^\top ,\\
\nonumber
 H\eq \epsilon_{N} \nabla_N\,N,\quad h=H\,g^\top ,\quad \<h,h\>=g(H,H), \\
 \tilde H\eq \epsilon_{N} \tilde\tau_1 N,\quad \tilde\tau_1= \epsilon_{N} \tr_{g}\tilde h_{sc},\quad
 \<\tilde h, \tilde h\> = \epsilon_{N} \tilde\tau_2,\quad
  {\rm Def}_{\widetilde{\cal D}}\,\tilde H = \epsilon_{N} N(\tilde\tau_1)\, g^\top \,.
\end{eqnarray}
 Notice that $(H^\flat\otimes H^\flat)(X,Y)=g(H,X)\,g(H,Y)$.
Substituting
 $\Phi_h=0=\Sm_{\,\rm ex}$,
 $\widetilde{\Sm}_{\,\rm ex}=\epsilon_{N} (\tilde\tau_1^2-\tilde\tau_2)$,
 $\widetilde{\cal T}=\epsilon_{N} \tilde T^{\sharp\,2}_N$
into \eqref{E-main-0i} and using \eqref{E-Jmix-dim1} yields~\eqref{E-main-1i}.
Substituting
 $h =H\,g^\top$,
 $\Phi_{\tilde h} = \epsilon_{N} (\tilde\tau_1^2-\tilde\tau_2)\, g^\top$,
 $\Phi_{\tilde T} = -\<\tilde T, \tilde T\>\, g^\top$,
 ${\cal K}^\flat =0$
into \eqref{E-main-0iii}
 and using \eqref{E-Jmix-dim1} yields \eqref{E-main-2i}.

Let $X$ be orthogonal to~$N$ with $\nabla_{Z}X \in \widetilde{\cal D}$ for all $Z \in TM$. We have $\theta=0$ and since
\begin{eqnarray*}
 2\,(\Div \alpha)(X,N) \eq g(\nabla_N H -\tilde \tau_1 H, X), \\
 2\,\<{\tilde\theta} - {\tilde\alpha}, H\>(X,N) \eq -g({\tilde T}^{\sharp}_N(H) + {\tilde A}_N(H), X), \\
  2\,{\rm Sym}(H^{\flat} \otimes H^{\flat})(X,N) \eq g({\tilde\tau}_1 H,\, X), \\
  2\,{\tilde \delta}_H (X,N) \eq g(\nabla_N H, X), \\
  2\,\Lambda_{\alpha, {\tilde\alpha}}(X,N) \eq g({\tilde A}_N(H), X),
\end{eqnarray*}
the Euler-Lagrange equation \eqref{E-main-0ii} reduces to
\begin{equation} \label{E-main-0ii-flows}
 ( \Div {\tilde \theta} )_{ | \widetilde{\cal D} \times {\cal D}} + ( \Div {\tilde \theta} )_{ |  {\cal D} \times \widetilde{\cal D} } - \< {\tilde \theta} , H \> =0.
\end{equation}
For $X \in \cal D$ such that $\nabla_{Z}X \in \widetilde{\cal D}$ for all $Z \in TM$ we have
\begin{eqnarray*}
 2\,\Div{\tilde\theta}(X,N) \eq \sum\nolimits_{\,i}\epsilon_i\,g((\nabla_{{\cal E}_i}{\tilde T}^{\sharp}_{N})(X), {\cal E}_i)
 +\epsilon_N g(\nabla_{N} ({\tilde T}^{\sharp}_{N} (X)), N) \\
 \eq ({\Div}^{\perp}\,{\tilde T}^{\sharp}_{N}) (X) +g({\tilde T}^{\sharp}_{N}(H), X). 
\end{eqnarray*}
Hence, \eqref{E-main-0ii-flows} can be written as \eqref{E-main-3i}.
\qed

\smallskip

By \eqref{E-divP}, we have $\Div \tilde h = N(\tilde h_{sc}) - \tilde \tau_1\tilde h_{sc}$ and $\Div h =(\Div H)\,\tilde g$.
 Then, see \eqref{E-genRicN} and \eqref{eq-ran-ex},
\begin{eqnarray}\label{E-RicNs1aa}
\nonumber
 \epsilon_{N} \big(R_N + \tilde A_N^2+(\tilde T^\sharp_N)^2\big)^\flat
 \eq N(\tilde h_{sc}) -H^\flat\otimes H^\flat +{\rm Def}_{\cal D}\,H,\\
 \epsilon_{N} \Ric_{N}
 \eq \Div H + \epsilon_{N} (N(\tilde\tau_1) -\tilde\tau_2) + \<\tilde T, \tilde T\>.
\end{eqnarray}
Remark that \eqref{E-RicNs1aa}$_2$ is simply the trace of \eqref{E-RicNs1aa}$_1$.

A flow of a unit vector $N$ is called \textit{geodesic} if the orbits are geodesics ($h=0$)
and \textit{Riemannian} if the metric is bundle-like ($\tilde h=0$).
A nonsingular Killing vector field clearly defines a Riemannian flow; moreover, a Killing vector field of constant length generates a geodesic Riemannian~flow. Restricting Corollary \ref{C-main02} to the case of a geodesic Riemannian~flow, we obtain the following.

\begin{cor}\label{P-flows}
Let $\widetilde{\cal D}$ be spanned by a unit vector field $N$ that generates a geodesic Riemannian flow on a pseudo-Riemannian manifold $(M^{p+1},g)$.
The metric $g$ is critical for the action \eqref{E-Jmix} with respect to ${\bar g}^\perp$-variations if and only if all the following conditions hold:
\begin{subequations}
\begin{eqnarray}\label{E-1geod-Riem}
 && R_N = (1/p)\,\Ric_{N}\id^\perp, \\ 
\label{geodriemflowiii}
 &&
 \Ric(X, N)= 0\quad ( X \in {\cal D} ) ,\\
 \label{RicNconstpneq4}
&& {\rm If}\ p\ne4 \; {\rm then} \; \Ric_{N}={\rm const}\ .
\end{eqnarray}
\end{subequations}
If $\Ric_{N}={\rm const}$ and \eqref{E-1geod-Riem}, \eqref{geodriemflowiii} are satisfied, then the metric is critical for the action \eqref{E-Jmix} with respect to ${\bar g}^\perp$-variations for all $\Omega \subset M$.

The metric $g$ is critical for the action \eqref{E-Jmix} with respect to ${\bar g}^\top$-variations if and only if
\begin{equation} \label{geodriemflowscalar}
\Ric_N = \const.
\end{equation}
The metric $g$ is then critical for the action \eqref{E-Jmix} with respect to ${\bar g}^\top$-variations for all $\Omega \subset M$.

\end{cor}

\begin{proof}
As was proved in \cite{rz-1}, conditions \eqref{E-1geod-Riem}, \eqref{RicNconstpneq4} and equation \eqref{geodriemflowscalar} follow from \eqref{E-main-1i}, \eqref{E-main-2i} and their traces. Condition for the metric to be critical for all $\Omega \subset M$ is $\Sm^*_{\rm mix} ={\rm const}$ or $\widetilde{\Sm}^*_{\rm mix} = {\rm const}$, respectively for ${\bar g}^\perp$- and ${\bar g}^\top$-variations.

For a geodesic Riemannian $N$-flow, \eqref{E-main-3i} reduces to condition ${\Div}^{\perp} {\tilde T}^{\sharp}_{N}(X) = 0$ for all $X \in {\cal D}$, that we shall now examine. A Riemannian geodesic flow locally gives rise to a Riemannian submersion with totally geodesic fibers. Such mappings can be described by the following tensor, introduced by Gray \cite{g1967} and here adjusted to our notation:
\[
 {\cal O}_X Y = ( \nabla_{X^{\perp}} Y^\top )^\perp + ( \nabla_{X^{\perp}} Y^\perp )^\top \quad
 (X,Y \in TM).
\]
It follows that $\cal O$ is antisymmetric with respect to $g$ and we have ${\cal O}_X Y = {\tilde T}(X,Y)$ for $X,Y \in {\cal D}$. Hence, for $X, Y \in {\cal D}$ we have $g( {\tilde T}^\sharp_N  X,Y) = g( {\tilde T}(X,Y) , N ) = g( {\cal O}_X Y , N ) = - g( {\cal O}_X N , Y )$ and we obtain ${\tilde T}^\sharp_N X = - {\cal O}_X N$.


Let $X \in {\cal D}$ and $\nabla_Z X \in {\widetilde D}$ for all $Z \in TM$. Using an adapted frame with ${\cal E}_i \in {\widetilde D}$ at a point, the fact that $\nabla_N N =0$, and the antisymmetry of $(\nabla_Z {\cal O})$ for all $Z \in TM$, we obtain:
\begin{eqnarray*}
({\Div}^{\perp} {\tilde T}^{\sharp}_{N})(X) \eq \sum\nolimits_i g(\nabla_{{\cal E}_i}{\tilde T}^\sharp_N X, {\cal E}_i)
= - \sum\nolimits_i g( \nabla_{ {\cal E}_i }  {\cal O}_X N , {\cal E}_i ) \\
 \eq - \sum\nolimits_i g( (\nabla_{ {\cal E}_i }  {\cal O})_X N , {\cal E}_i )
 = \sum\nolimits_i g(  ( \nabla_{ {\cal E}_i } {\cal O} )_X {\cal E}_i , N ) .
\end{eqnarray*}
From the formula (5.37e) from \cite{to2}, 
adjusted to our definitions of $R$ and $\Ric$, it follows that
\[
 ({\Div}^{\perp}{\tilde T}^{\sharp}_{N})(X) = -\sum\nolimits_i g(R({\cal E}_i, X) {\cal E}_i, N) = -\Ric(X, N).
\]
Thus, we obtain \eqref{geodriemflowiii}.
\end{proof}

From Corollary \ref{P-flows} we immediately obtain the following.

\begin{cor} \label{EinsteinKilling}
Let $(M^{p+1}, g)$, with $p>1$, be an Einstein manifold with a geodesic Riemannian flow. Let $\widetilde{\cal D}$ be the $1$-dimensional distribution tangent to the flowlines. Then $g$ is critical for the action \eqref{E-Jmix} with respect to both ${\bar g}^\perp$ and ${\bar g}^\top$-variations, for all $\Omega \subset M$.
\end{cor}

In \cite{rz-1} it was also proved that for metrics critical with respect to ${\bar g}^\perp$-variations if~$p$ is odd then $\Ric_{N}=0$ and $M$ splits, and if $\Ric_{N}\ne0$ then $p$ is even (and moreover, for $p\ne4$, $\Ric_{N}$ is a function of a point only).

The following proposition shows that manifolds with geodesic Riemannian flows critical for the action \eqref{E-Jmix} with respect to all volume-preserving variations are in fact metric products. One can similarly prove this statement for arbitrary variations of the metric.

\begin{prop}
Let $\widetilde{\cal D}$ be spanned by a unit vector field $N$ that generates a geodesic Riemannian flow on a pseudo-Riemannian manifold $(M^{p+1},g)$. If $g$ is a critical metric for the action \eqref{E-Jmix} with respect to all volume preserving variations, then ${\cal D}$ is integrable.
\end{prop}

\begin{proof}
Using Remark \ref{remarkAllVariations}, we can write the Euler-Lagrange equations for arbitrary volume preserving variations as follows:
\begin{eqnarray}
\label{E-main-1K}
 && \epsilon_{N} \big(R_N -(\tilde T^\sharp_N)^2 \big)^\flat = \frac12\,\big(\epsilon_{N} \Ric_{N} - \lambda \big)\,g^\perp
,\\
&& \Ric(X, N)= 0\quad ( X \in {\cal D} )
,\\
\label{E-main-2K}
 && \epsilon_{N} \Ric_{N} =  4\,\<\tilde T, \tilde T\> - \lambda ,
\end{eqnarray}
where $\lambda \in C^\infty(M)$ is an arbitrary function. From Proposition \ref{L-CC-riccati} we obtain $R_N =-(\tilde T^\sharp_N)^2$ and hence the trace of \eqref{E-main-1K} reads
\[
\lambda = \frac{p-4}{p} \< {\tilde T} , {\tilde T} \>.
\]
On the other hand, from \eqref{E-main-2K} we obtain $\lambda = 3 \< {\tilde T} , {\tilde T} \>$. The two equations for $\lambda$ have a solution only for $\< {\tilde T} , {\tilde T} \> =0$.
\end{proof}


\subsection{Contact metric structures} \label{SectionContact}

Contact manifolds come with a natural foliation given by the flowlines of the Reeb field. They also admit an (in general, non-unique) \textit{associated metric} of well examined properties - we show that for such metric one of the Euler-Lagrange equations, \eqref{E-main-0ii}, always holds. Then we examine the remaining Euler-Lagrange equations for ${\bar g}^\perp$- and ${\bar g}^\top$-variations, which generalize variations among associated metrics considered in \cite{b2010}.

Recall \cite{b2010} that a manifold $M^{2n+1}$ with a $1$-form $\eta$ such that
\[
 d\eta(\xi,X)=0\quad (X\in TM),\quad \eta(\xi)=1,
\]
is called a \textit{contact manifold}, and $\xi$ is called the \textit{characteristic vector field} (or the \textit{Reeb field}).
A Riemannian metric $g$ on a contact manifold $(M^{2n+1},\eta)$ is \textit{associated}
if there exists a $(1,1)$-tensor $\phi$ such that for all $X,Y \in TM$
\begin{eqnarray}\label{E-contact-metric1}
 \eta(X) = g(\xi,X) , \quad 
 d\eta(X,Y) = g(X,\phi(Y)),\quad \phi^2=-I +\eta\otimes\xi . 
\end{eqnarray}
The above $(\phi, \xi, \eta, g)$ is called a \textit{contact metric structure} on $M$.
For all contact manifolds we consider, let $\widetilde{\cal D}$ be spanned by $\xi$ and let ${\cal D}$ denote its orthogonal complement.

\begin{rem} \label{remarkpseudocontact}
\rm
While we shall consider only the Riemannian metric in this and the next section, there is a natural way to make a Riemannian contact manifold $(M,\eta, g)$ a pseudo-Riemannian contact manifold: by setting $g - 2\,\eta \otimes \eta$ as the new metric \cite{brunettipastore}. This transformation does not invalidate our main results: Proposition \ref{Kcontactarecritical} and Corollary \ref{contactTcritical}.
\end{rem}

\begin{prop} \label{E-main-3i-contact-metric-structure}
Let $(\phi, \xi, \eta, g)$ be a contact metric structure on $M$. Then \eqref{E-main-3i} is satisfied for $N = \xi$.
\end{prop}

\begin{proof}
For a contact metric structure we have (see \cite{b2010})
\begin{equation}\label{H0contact}
 H = \nabla_{\xi}\,\xi =0.
\end{equation}
For all $X,Y$ such that $g(X,\xi) = g(Y, \xi) =0$
\[
 d\eta (X,Y) = - \frac{1}{2}\,\eta([X,Y]) = - \frac{1}{2}\,g([X,Y],\xi)\,\eta(\xi) = - \frac{1}{2}\,g([X,Y],\xi) = - g({\tilde T}^{\sharp}_{\xi}(X),Y).
\]
Hence, it follows from \eqref{E-contact-metric1}, that
\[
 g(X, \phi(Y)) = g(X, {\tilde T}^{\sharp}_{\xi}(Y) ),
\]
and since $\phi (\xi) =0$, we obtain the following equality 
\[
 \,{\tilde T}^{\sharp}_{\xi} = \phi.
\]
By \eqref{H0contact}, \eqref{E-main-3i} for $N=\xi$ reduces to
 ${\Div}^{\perp} ({\tilde T}^{\sharp}_{\xi}) |_{\cal D} = 0$,
which takes the following form:
\begin{equation*} 
\forall_{Y \in {\cal D}} \quad ({\Div}^{\perp}\phi)(Y) =0 .
\end{equation*}
For $Y \in \cal D$, the formula for contact metric structures in \cite[Corollary~6.1]{b2010} yields
\[
 2\,g((\nabla_{{\cal E}_i}\,\phi)(Y), {\cal E}_i) = g([\phi, \phi](Y, {\cal E}_i), \phi({\cal E}_i)),
\]
where
\[
 [\phi, \phi](X,Y) = \phi^{2}[X,Y] + [\phi(X), \phi(Y)] - \phi[\phi(X), Y] - \phi[ X, \phi(Y)] .
\]
As in \cite[Corollary~6.1]{b2010}, considering an orthonormal $\phi$-basis (see~\cite[p.~44]{b2010}), i.e. assuming that ${\cal E}_{i+p/2} = \phi ( {\cal E}_i )$ for $i=1 , \ldots , p/2$, we obtain that
\[
 \sum\nolimits_{i=1}^p g([\phi, \phi](Y, {\cal E}_i), \phi({\cal E}_i))
 = -\sum\nolimits_{i=1}^p g([\phi,\phi](Y, \phi({\cal E}_i)), \phi^2({\cal E}_i)) .
\]
Hence, $({\Div}^{\perp}\phi)(Y)=0$ and \eqref{E-main-3i} is satisfied.
\end{proof}

Note that condition \eqref{E-main-3i} depends on both the metric and the distribution $\cal D$. The following corollary gives some idea about its stability under a conformal change of the metric.

\begin{prop} \label{confchangedivtheta}
Let $(\phi, \xi, \eta, g)$ be a contact metric structure on
$M$. Let $\psi \in C^{\infty}(M)$ be a nonconstant function on $M$. Then on $(M, \bar g )$, where $\bar g = e^{-2 \psi} g$, equation \eqref{E-main-3i} is not satisfied.
\end{prop}

\proof
Let $\bar \nabla$ be the Levi-Civita connection on $(M, \bar g )$. Then for all $X,Y \in TM$ we have
\[
{\bar \nabla}_X Y = \nabla_{X}Y - (X \psi)(Y) - (Y\psi)(X) + g(X,Y) \nabla \psi,
\]
and also for any $(1,1)$-tensor $S$
\begin{eqnarray*}
({\bar \nabla}_X S)(Y) \eq {\bar \nabla}_X S(Y) - S( {\bar \nabla}_X Y ) \\
\eq \nabla_{X} S(Y) - (X \psi) S(Y) - (S(Y)\psi)\,X + g(X, S(Y)) \nabla \psi \\
&& - S(\nabla_X Y) + ( X \psi )S(Y) + (Y\psi) S(X) - g(X,Y) S(\nabla\psi) \\
\eq ({\nabla}_X S)(Y) - (S(Y)\psi) X + g(X, S(Y)) \nabla \psi + (Y\psi) S(X) - g(X,Y) S(\nabla\psi).
\end{eqnarray*}
On $(M, {\bar g})$ equation \eqref{E-main-3i} takes the following form:
\begin{equation} \label{E-main-3i-barg}
 (\Div_{\bar g}^{\perp} { {\tilde T}_{ \bar N }^{\sharp}})(Y) + 2\,{\bar g}( {\tilde T}_{ \bar N }^{\sharp} H_{\bar g} , Y ) =0,
\end{equation}
where $\bar N$ is the $\bar g$-unit normal vector along $\widetilde{\cal D}$ and ${\tilde T}^\sharp$ is the action of the ``musical isomorphism" with respect to $\bar g$ on $\tilde T$.


Since ${\bar g}(\xi , \xi) = e^{-2\psi} g(\xi ,\xi)=e^{-2 \psi}$, we see that ${\bar N} = e^{\psi} \xi$.
Note that ${\tilde T}(X,Y) = \frac{1}{2}\,[X,Y]^{\top}$ ($X,Y \in \cal D$) is invariant with respect to a conformal change of metric; it also follows that the operators ${\tilde T}^\sharp$ defined with respect to $g$ and $\bar g$ coincide. We have
\[
 2\,{\bar g}( {\tilde T}^{\sharp}_{ \bar N } X, Y ) = {\bar g}( [X,Y], \bar N ) =
 {\bar g}( [X,Y], e^{\psi} \xi ) = 2\,e^{\psi} {\bar g}( {\tilde T}^{\sharp}_{ \xi } X, Y ),
\]
and so ${\tilde T}^{\sharp}_{ \bar N } = e^{\psi} {\tilde T}^{\sharp}_{ \xi }$.

Let ${\cal E}_i$ be a $g$-orthonormal basis of $\cal D$. Then $e^\psi {\cal E}_i$ is a $\bar g$-orthonormal basis of $\cal D$.
We have
\begin{eqnarray*}
(\Div_{\bar g}^{\perp} { {\tilde T}_{ \bar N }^{\sharp}})(Y) \eq (\Div_{\bar g}^{\perp} ( e^{\psi} {\tilde T}_{ \xi }^{\sharp} ) )(Y) \\
 \eq e^{\psi} (\Div_{\bar g}^{\perp} { {\tilde T}_{ \xi }^{\sharp}})(Y)
 +\sum\nolimits_{i} \epsilon_i e^{\psi} {\cal E}_i ( e^{\psi} ) \cdot {\bar g}( {\tilde T}_{ \xi }^{\sharp} (Y) , e^{\psi } {\cal E}_i  )  \\
 \eq e^{\psi} (\Div_{\bar g}^{\perp} { {\tilde T}_{ \xi }^{\sharp}})(Y) + e^{\psi } g( {\tilde T}_{ \xi }^{\sharp} Y , \nabla \psi  )
\end{eqnarray*}
and
\begin{eqnarray*}
 (\Div_{\bar g}^{\perp} {{\tilde T}_{\xi}^{\sharp}})(Y) \eq \sum\nolimits_{i} {\bar g}(({\bar \nabla}_{e^\psi {\cal E}_i}{\tilde T}_{\xi}^{\sharp} )(Y), e^{\psi} {\cal E}_i ) = \sum\nolimits_{i} g( ({\bar \nabla}_{ i } {\tilde T}_{\xi}^{\sharp} ) (Y) , {\cal E}_i ) \\
 \eq (\Div^{\perp} {\tilde T}_{\xi}^{\sharp})(Y) + \sum\nolimits_{\,i} \big( - ( {\tilde T}_{\xi}^{\sharp}Y \psi) g( {\cal E}_i, {\cal E}_i )
 + g({\cal E}_i, {\tilde T}_{\xi}^{\sharp} Y ) g( \nabla \psi, {\cal E}_i ) \\
 && +\,(Y\psi) g({\tilde T}_{\xi}^{\sharp}{\cal E}_i, {\cal E}_i) - g(Y, {\cal E}_i) g({\tilde T}_{\xi}^{\sharp}\nabla\psi, {\cal E}_i)\big) \\
 \eq - p ( {\tilde T}_{\xi}^{\sharp} Y \psi ) + g( \nabla \psi , {\tilde T}_{\xi}^{\sharp} Y ) - g( Y , {\tilde T}_{\xi}^{\sharp}\nabla\psi) \\
 \eq - p ( {\tilde T}_{\xi}^{\sharp} Y \psi ) + 2g( \nabla \psi , {\tilde T}_{\xi}^{\sharp} Y )
 = (2-p)\,g( \nabla \psi, {\tilde T}_{\xi}^{\sharp} Y )
\end{eqnarray*}
for all $Y \in \cal D$.
 On the other hand, $H_{\bar g} = ({\bar \nabla}_{ \bar N } { \bar N } )^{\perp}$; hence,
\begin{eqnarray*}
 H_{\bar g} \eq ({\bar \nabla}_{e^{\psi} \xi} e^{\psi} \xi)^{\perp} = \big( e^{2\psi} {\bar \nabla}_{\xi} \xi + e^{\psi} (\xi e^{\psi}) \xi \big)^{\perp} \\
 \eq e^{2\psi} (H - (2(\xi \psi) \xi)^{\perp} + g(\xi , \xi)(\nabla \psi)^{\perp}) = e^{2\psi} (\nabla \psi)^{\perp},\\
 {\bar g}({\tilde T}_{ \bar N }^{\sharp} H_{\bar g}, Y ) \eq  e^{-2\psi}g( {\tilde T}_{ e^{\psi } \xi}^{\sharp}( e^{2\psi} \nabla \psi), Y)\\
 \eq e^{\psi} g( {\tilde T}_{ \xi}^{\sharp}( \nabla \psi), Y )
 = - e^{\psi } g( {\tilde T}_{ \xi}^{\sharp} (Y), \nabla \psi  ).
\end{eqnarray*}
Using all the above, we can write the left-hand side of \eqref{E-main-3i-barg} as
\begin{eqnarray*}
 &&\hskip-8mm \frac{1}{2}\,(\Div_{\bar g}^{\perp} {{\tilde T}_{\bar N}^{\sharp}})(Y)
 +{\bar g}({\tilde T}_{\bar N}^{\sharp} H_{\bar g}, Y)\\
 && = \frac{1}{2}\,e^{\psi}(2-p)\,g(\nabla \psi, {\tilde T}_{\xi}^{\sharp} Y)
  +\frac{1}{2}e^{\psi}\,g(\nabla \psi, {\tilde T}_{\xi}^{\sharp} Y) - e^{\psi} g({\tilde T}_{ \xi}^{\sharp}(Y), \nabla\psi) \\
 && = \big( \frac{2-p}{2} + \frac{1}{2} - 1 \big) e^{\psi}\,g({\tilde T}_{\xi}^{\sharp} (Y), \nabla\psi)
 = \frac{1-p}{2} e^{\psi}\,g({\tilde T}_{\xi}^{\sharp}(Y), \nabla \psi ).
\end{eqnarray*}
Note that $p\ge2$ on contact metric manifolds, and since the image of ${\tilde T}_{\xi}^{\sharp} = -\phi$ is the whole ${\cal D}$, for any $\psi$ such that $(\nabla \psi)^{\perp}\neq 0$ there exists $Y\in\cal D$
such that $g({\tilde T}_{\xi}^{\sharp}(Y), \nabla\psi)\ne0$ and \eqref{E-main-3i} is not satisfied. Finally, if $(\nabla \psi)^{\perp} = 0$, then from the fact that $\cal D$ is bracket generating, it follows that $\psi$ is in fact constant on $M$.
\qed

\smallskip

On any contact manifold there exists a (non-unique) contact metric structure, see \cite{b2010}.
Among them there is a class particularly interesting from the geometric point of view.

\begin{defn}[\cite{b2010}]\rm
A contact metric structure for which $\xi$ is Killing is called $K$-\textit{contact}.
\end{defn}

\begin{prop} \label{Kcontactarecritical}
Any $K$-contact metric $g$ is critical for the action \eqref{E-Jmix-Ncontact}, with respect to both ${\bar g}^\perp$- and ${\bar g}^\top$-variations, for all $\Omega \subset M$.
\end{prop}

\begin{proof}
We have already seen in \eqref{H0contact} that the integral curves of $\xi$ are geodesics for the contact metric structure. On the other hand, a nonsingular Killing vector field defines a Riemannian flow ($\tilde h=0$). Thus, in case of a $K$-contact structure, we can use Proposition \ref{P-flows}.

By \cite[Theorem~7.2]{b2010}, if $(M,g)$ is a $K$-contact manifold then \eqref{E-1geod-Riem} is satisfied with $\Ric_{N} = p$.
As was shown in Proposition~\ref{E-main-3i-contact-metric-structure}, also \eqref{geodriemflowiii} holds. The last statement follows from the fact that both $\Sm^*_{\rm mix}$ and $\widetilde{\Sm}^*_{\rm mix}$ are constant.
\end{proof}

In \cite{b2010}, the action \eqref{E-Jmix}, which reduces to 
\begin{equation}\label{E-Jmix-Ncontact}
 J_{\rm mix,\widetilde{\cal D},\Omega}:\ g \rightarrow  \int_{\Omega}\Ric_{N} (g) \,{\rm d}\vol_g
\end{equation}
has been studied on the set of metrics associated to a given contact form.

\begin{defn}[\cite{b2010}, p.~24]\rm
A contact structure is \textit{regular} if $\xi$ is regular as a vector field, that is, every point of the manifold has a neighborhood such that any integral curve of $\xi$ passing through the neighborhood passes through only once.
\end{defn}

\begin{thm}[see Theorem~10.12 in \cite{b2010}] \label{blair1012}
An associated metric $g$ on a compact regular contact manifold $(M, \eta)$ is critical for the action \eqref{E-Jmix-Ncontact} considered on the set of metrics associated to $\eta$ if and only if it is $K$-contact.
\end{thm}

We have $g(\xi, \xi)=1$ for any associated metric and the volume form of associated metric on a contact manifold can be expressed only in terms of $\eta$ and $d \eta$. Therefore, variations of the metric restricted to the set of all associated metrics form a subclass of the volume preserving $g^\perp$-variations (which form a subclass of ${\bar g}^\perp$-variations).
Hence, on compact regular contact manifolds Proposition \ref{Kcontactarecritical} and Theorem \ref{blair1012} together give the following characteristic of some critical metrics -- for a larger space of variations.

\begin{cor}
Let $(M, \eta)$ be a compact regular contact manifold and let $g$ be an associated metric. Then $g$ is critical for the action \eqref{E-Jmix-Ncontact} for ${\bar g}^\perp$-variations if and only if $g$ is $K$-contact.
\end{cor}

The above result indicates that while \eqref{E-main-3i} holds for all contact metric structures, on compact regular manifolds only $K$-contact structures among them satisfy the remaining Euler-Lagrange equations \eqref{E-main-1i} and \eqref{E-main-2i}. This may be the case also on non-compact manifolds, as for a contact metric structure on $\RR^{3}$ defined in \cite[pp.~121--122]{b2010}, considered in the following example.

\begin{example} \rm
Let
\[
 \eta = \frac{1}{2}\,(dz - y\,dx), \quad
 g = \frac{1}{4} \bigg(\begin{array}{ccc}
   1 + y^2 + z^2 & z & -y \\
   z & 1 & 0 \\
   -y & 0 & 1
 \end{array}\bigg).
\]
Using an adapted orthonormal frame:
 $\xi=2\,\frac{\partial}{\partial z}, E_1 = 2(\frac{\partial}{\partial x} -z\,\frac{\partial}{\partial y} +y\,\frac{\partial}{\partial z}),E_2
 = 2\,\frac{\partial}{\partial y}$,
 one can show that in $\{E_1, E_2\}$ basis of ${\cal D}$ we have
\[
 \tilde A = \bigg(\begin{array}{cc}
   0 & -1 \\
   -1 & 0
 \end{array}\bigg),
 \quad
 {\tilde T}^{\sharp}_{\xi} =\bigg(\begin{array}{cc}
   0 & 1 \\ -1 & 0 \end{array}\bigg).
\]
We have $\Ric_{N} = 0$ \cite{b2010}, $H=0$, ${\tilde \tau}_1=0$ and $\big(R_N +\tilde A_N^2 -(\tilde T^\sharp_N)^2 +[\tilde T_N^\sharp,\tilde A_N]\big)^\flat$ is not conformal, hence \eqref{E-main-1i} is not satisfied, although both \eqref{E-main-3i} and \eqref{E-main-2i} hold.

\end{example}

Flowlines of Reeb vector fields on contact manifolds are often described as having ``maximally non-integrable" orthogonal distributions.
We can give this notion a precise meaning by conside\-ring the following action:
\begin{equation} \label{JtildeT}
 J_{{\tilde T},\Omega}:\ g \rightarrow  \int_{\Omega}\< \tilde T, \tilde T \>\,{\rm d}\vol_g ,
\end{equation}
and showing that contact metric structures are its critical points. Note that \eqref{JtildeT} is the total norm of the integrability tensor of the (varying) orthogonal complement of a fixed distribution $\widetilde{\cal D}$.

\begin{prop}[\bf Euler-Lagrange equations]
A metric $g\in{\rm Riem}(M,\widetilde{\cal D},{\cal D})$ is critical for the action \eqref{JtildeT} with respect to ${\bar g}^\perp$-variations if and only if
\begin{eqnarray} \label{ELtildeT1}
 && 2\,\widetilde{\cal T}^{\flat} =-\big(\frac{1}{2}\,\< \tilde T, \tilde T \> + {\tilde T}^*(\Omega, g)\big)\,g^{\perp}\ 
 , \\
 \label{ELtildeT2}
&& \Lambda_{\tilde\theta, \theta - \alpha} = ( \Div {\tilde\theta} )_{ | \widetilde{\cal D} \times {\cal D}} + ( \Div {\tilde \theta} )_{ |  {\cal D} \times \widetilde{\cal D} } , 
\end{eqnarray}
where
\[
 {\tilde T}^* =  \frac{4-p}{2\,p}\,\< \tilde T, \tilde T \>,\qquad
\]
A metric $g\in{\rm Riem}(M,\widetilde{\cal D},{\cal D})$ is critical for the action \eqref{JtildeT} with respect to ${\bar g}^\top$-variations if and only if
\begin{eqnarray}
\label{ELtildeT3}
&& \Phi_{\tilde T} = \big(\frac{1}{2}\,\<\tilde T,\tilde T\> - T^*(\Omega, g)\big)\,\tilde g ,
 \ \ ({\rm for}~g^\perp\mbox{\rm-variations}),
\end{eqnarray}
where
\[
 T^* = \frac{2+n}{2\,n}\,\< \tilde T, \tilde T \> .
\]
\end{prop}

\begin{proof}
Let $g_t$ be a $g^\perp$-variation and ${\bar g}_t$ be related to it by \eqref{e:Var-Bar}. The norms of integrability tensor of ${\cal D}$ with respect to $g_t$ and ${\bar g}_t$ are related as follows \cite{rz-1}: $\<\tilde T,\tilde T\>_{\bar g}=(\phi_t^\perp)^{-2} \<\tilde T,\tilde T\>_{g}$. Using $\phi_0^\perp =1$ and $\dt \phi_t^\perp =-\frac{\phi_t^\perp}{p}\,(\tr_{g_t} {B}_t)({\Omega}, g_t)$, we get the following formula
\begin{eqnarray*}
 {\rm\frac{d}{dt}}\,J_{\tilde T,\Omega}(\bar g_t)_{\,|\,t=0}
 \eq {\rm\frac{d}{dt}}\,J_{\tilde T,\Omega}(g_t)_{\,|\,t=0}
 +\frac{4-p}{2\,p}\,\<\tilde T, \tilde T\>(\Omega, g)\int_{\Omega} (\tr B) \,{\rm d}\vol_g ,\\
\end{eqnarray*}
and use Proposition \ref{propvar1} to obtain
\begin{eqnarray*}
 {\rm\frac{d}{dt}}\,J_{\tilde T,\Omega}(g_t)_{\,|\,t=0} \eq \int_{\Omega} \<\,2\widetilde{\cal T}^{\flat}
 +2\Lambda_{\tilde \theta, \theta - \alpha} - 2( \Div{\tilde \theta} )_{ | \widetilde{\cal D} \times {\cal D}} - 2 ( \Div {\tilde \theta} )_{ |  {\cal D} \times \widetilde{\cal D} } + \frac{1}{2}\<\tilde T, \tilde T,\, \>\ g^{\perp} , B\>\,{\rm d}\vol_g.
\end{eqnarray*}
Decomposing the resulting Euler-Lagrange equation into parts defined on
${\cal D}\times{\cal D}$ and ${\cal D}\times\widetilde{\cal D}$ yields \eqref{ELtildeT1} and \eqref{ELtildeT2}.
For $g^\perp$-variation $g_t$ and corresponding to it ${\bar g}^\perp$-variation ${\bar g}_t$, we can analogously obtain variation formulas for $J_{T,\Omega}(g) = \int_{\Omega}\<T, T\>\,{\rm d}\vol_g$ (now we have $\<T,T\>_{\bar g}=\phi\,\<T,T\>_{g}$):
\begin{eqnarray*}
 {\rm\frac{d}{dt}}\,J_{T,\Omega}(\bar g_t)_{\,|\,t=0} \eq {\rm\frac{d}{dt}}\,J_{T,\Omega}(g_t)_{\,|\,t=0}
  -\frac{p+2}{2\,p}\,\< T, T\>(\Omega,g)\int_{\Omega}\<g^{\perp},\ B\>\,{\rm d}\vol_g,\\
 {\rm\frac{d}{dt}}\,J_{T,\Omega}(\bar g_t)_{\,|\,t=0} \eq -\int_{\Omega}\<\Phi_T - \frac{1}{2}\,\<T,T\>g^{\perp},\ B\>\,{\rm d}\vol_g.
\end{eqnarray*}
Then we obtain \eqref{ELtildeT3} as a formula dual to one obtained for ${\cal D}$-variations of $J_{T,\Omega}(g)$.
\end{proof}

Note that, as expected, distributions with integrable orthogonal complement are critical for~\eqref{JtildeT}.
Taking traces of \eqref{ELtildeT1} and \eqref{ELtildeT3}, we obtain the following.

\begin{cor}
If a metric $g$ is critical for the action \eqref{JtildeT} with respect to both ${\bar g}^\perp$-variations and ${\bar g}^\top$-variations,
then $\<\tilde T, \tilde T\> = \const$.
\end{cor}

Also, using the results obtained for the contact metric structure, we get the following.

\begin{prop} \label{contactTcritical}
Let $(\phi, \xi, \eta, g)$ be a contact metric structure on $M$ and let $\widetilde{\cal D}$ be spanned by~$\xi$.
Then $g$ is critical for the action \eqref{JtildeT} with respect to both ${\bar g}^\perp$-variations and ${\bar g}^\top$-variations.
\end{prop}

\begin{proof}
Using results from the proof of Proposition \ref{E-main-3i-contact-metric-structure}, we compute
\begin{eqnarray*}
 \<\tilde T, \tilde T\> \eq \sum\nolimits_{\,i,j} g( {\tilde T}({\cal E}_i, {\cal E}_j ), {\tilde T}({\cal E}_i, {\cal E}_j)) \\
 \eq \sum\nolimits_{\,i,j} g({\tilde T}^{\sharp}_{\xi} ({\cal E}_i), {\cal E}_j )^{2}
 = \sum\nolimits_{\,i} g(\phi({\cal E}_i), \phi({\cal E}_i))^{2} = p.
\end{eqnarray*}
We also have $\widetilde{\cal T} = ({\tilde T}^{\sharp}_{\xi})^{2} = - \id$; hence,
${\widetilde{\cal T}}^{\flat}=-\,g^{\perp}$.

By the above, \eqref{ELtildeT1} and \eqref{ELtildeT3} are satisfied;
\eqref{ELtildeT2} reduces to $ ( \Div {\tilde \theta} )_{ | \widetilde{\cal D} \times {\cal D}} + ( \Div {\tilde \theta} )_{ |  {\cal D} \times \widetilde{\cal D} } =0$ and holds by Proposition~\ref{E-main-3i-contact-metric-structure}.
\end{proof}

Contact metric structures are also example that the metrics critical for the action \eqref{JtildeT} with respect to both ${\bar g}^\perp$- and ${\bar g}^\top$-variations may not be critical with respect to all volume preserving variations of the metric. 

\begin{prop}
Let $(\phi, \xi, \eta, g)$ be a contact metric structure on $M^{p+1}$ and let $\widetilde{\cal D}$ be spanned by~$\xi$. The metric $g$ is critical for the action \eqref{JtildeT} with respect to all variations of $g$ that preserve the volume of $\Omega$ if and only if $p=2$.
\end{prop}

\begin{proof}
From Remark \ref{remarkAllVariations}, \eqref{ELtildeT1} and \eqref{ELtildeT2} we obtain the following Euler-Lagrange equations for arbitrary volume preserving variation of $g$:
\begin{eqnarray} \label{ELtildeT1gen}
 && 2\,\widetilde{\cal T}^{\flat} =-\big(\frac{1}{2}\,\< \tilde T, \tilde T \> + \lambda \big)\,g^{\perp}\
 , \\
 \label{ELtildeT2gen}
&& \Lambda_{\tilde\theta, \theta - \alpha} = ( \Div{\tilde\theta} )_{ | \widetilde{\cal D} \times {\cal D}} + ( \Div {\tilde \theta} )_{ |  {\cal D} \times \widetilde{\cal D} }
  , \\
\label{ELtildeT3gen}
&& \Phi_{\tilde T} = \big(\frac{1}{2}\,\<\tilde T,\tilde T\> - \lambda \big)\,\tilde g ,
\end{eqnarray}
where $\lambda \in C^\infty(M)$ is an arbitrary function. As in Proposition \ref{contactTcritical}, \eqref{ELtildeT2gen} is satisfied by a contact metric structure. Using ${\widetilde{\cal T}}^{\flat}=-\,g^{\perp}$ and $\Phi_{\tilde T} = -p g^\top$, we obtain from \eqref{ELtildeT1gen} and \eqref{ELtildeT3gen} the following system of equations
\[
\lambda = 4 - \frac{p}{2} , \quad \lambda = \frac{3}{2}p,
\]
that only has a solution for $p=2$.
\end{proof}

\subsection{Sasakian 3-structures} \label{section3Sasakian}

The Euler-Lagrange equations \eqref{E-main-1i} and \eqref{E-main-2i} have been obtained in \cite{rz-1} for variations of metric that preserved the orthogonality of distributions $\widetilde{\cal D}$ and $\cal D$. The new difficulty in finding examples of metrics critical for the action \eqref{E-Jmix} is condition \eqref{E-main-3i}.
For contact metric structures, it is satisfied due to \eqref{E-contact-metric1}$_2$, which is generalized by contact $3$-structures,
defined as follows \cite{kashiwada}.

\begin{defn}
A \emph{contact 3-structure} is defined as a set of three contact structures, $\eta_a , a =1 ,2,3$, with the same associated metric $g$ satisfying
\[
\phi_c = \phi_a \circ \phi_b - \eta_a \otimes \xi_b = - \phi_b \circ \phi_a + \eta_b \otimes \xi_a
\]
for any cyclic permutation $(a,b,c)$ of $(1,2,3)$. If each of them is Sasakian structure, it is called a \emph{Sasakian 3-structure} (some authors call it a \emph{3-Sasakian structure} ).
\end{defn}


\begin{thm}[\cite{kashiwada}]
 A contact 3-structure is necessarily a Sasakian 3-structure.
\end{thm}

For any Sasakian 3-structure, let $\widetilde{\cal D}$ be the distribution spanned by 3 characteristic vector fields $\xi_1 , \xi_2 , \xi_3$, its orthogonal complement will be denoted by $\cal D$. Since $ [\xi_a , \xi_b ] = \frac{1}{2} \xi_c $ for any cyclic permutation $(a,b,c)$ of $(1,2,3)$, we see that $\widetilde{\cal D}$ is integrable.

\begin{prop} \label{3Sasakianiscritical}
The metric of a Sasakian 3-structure on $M$ is critical for the action \eqref{E-Jmix} (where $\widetilde{\cal D}$ is spanned by the characteristic vector fields), with respect to both ${\bar g}^\perp$- and ${\bar g}^\top$-variations, for all $\Omega \subset M$.
\end{prop}

\begin{proof}
Since every $\xi_a$ defines a Sasakian structure, we can use the following formulas for any unit vectors $X,Y$ orthogonal to $\xi_a$ (so we can also have $X=\xi_b$ etc.):
\begin{eqnarray*}
  R(X,Y) \xi_a = \eta_a (Y) X - \eta_a (X) Y,  \quad
  R(X, \xi_a ) Y = - g(X,Y) \xi_a + \eta_a (Y) X .
\end{eqnarray*}
The above formulas are consistent with their analogues for $\xi_b$ and $\xi_c$, and yield the following:
\[
 r_{\cal D} = 3g^{\perp} , \quad r_{\widetilde{\cal D}} = p g^\top .
\]
We also have
\begin{eqnarray*}
 \Phi_{\tilde T} (\xi_a , \xi_b ) = {\widetilde \Psi}( \xi_a, \xi_b ) = - p g(\xi_a , \xi_b) ,\\
 {\tilde {\cal T}}^{\flat}(X,Y) = - 3g(X,Y),\quad (X,Y \in \cal D),
\end{eqnarray*}
and $\< \tilde T , \tilde T \> = 3p$. It follows that \eqref{E-main-0i} and \eqref{E-main-0iii} are satisfied, regardless of the choice of $\Omega$. As for any contact metric structure, we have $\nabla_{\xi_a } \xi_a= 0$ and $\phi_a ( \xi_a) =0$. Any vector $X \in \cal D$ is orthogonal to $\xi_a$ and in all tensor formulas we can assume that $\nabla_Z X $ is colinear with $\xi_a$ for all $Z \in TM$. Then
\begin{eqnarray*}
 (\Div {\tilde \theta} )(X, \xi_a ) \eq \sum_i g( ( \nabla_{ {\cal E}_i } \phi_a )(X) , {\cal E}_i ) + g( ( \nabla_{ \xi_b } \phi_a )(X) , \xi_b ) \\ && + g( ( \nabla_{ \xi_c } \phi_a )(X) , \xi_c ) + g( ( \nabla_{ \xi_a } \phi_a )(X) , \xi_a ) \\
 \eq \sum_i g( ( \nabla_{ {\cal E}_i } \phi_a )(X) , {\cal E}_i ) + g( ( \nabla_{ \xi_b } \phi_a )(X) , \xi_b ) \\ && + g( ( \nabla_{ \xi_c } \phi_a )(X) , \xi_c ) - g( \phi_a (X) , \nabla_{ \xi_a }  \xi_a ) \\
 \eq \sum_i g( ( \nabla_{ {\cal E}_i } \phi_a )(X) , {\cal E}_i ) + g( ( \nabla_{ \xi_b } \phi_a )(X) , \xi_b ) + g( ( \nabla_{ \xi_c } \phi_a )(X) , \xi_c ).
\end{eqnarray*}
and similarly for $\xi_b, \xi_c$. Since $(\phi_a, \xi_a, \eta_a, g)$ is a contact metric structure, it follows from Proposition~\ref{E-main-3i-contact-metric-structure} that \eqref{E-main-0ii} holds.
\end{proof}



There exist various ``extended theories of gravity" \cite{capoziello}, which investigate the original Ein\-stein-Hilbert action modified by adding to it terms of higher order, depending on the curvature. In the same vein, one may consider the following action
\begin{equation} \label{E-Jeps}
 J_\eps : g \mapsto \int_{\Omega} ( S(g) + \eps S_{\rm mix} (g) ) {\rm d }\vol_g ,
\end{equation}
where $S(g)$ is the scalar curvature of metric $g$. Clearly, \eqref{E-Jeps} is a
perturbation of the Einstein-Hilbert action by \eqref{E-Jmix}, and $\eps=-1$ yields the total sum of scalar curvatures of the distributions.
The action \eqref{E-Jeps} requires distinguishing a priori the distribution $\widetilde{\cal D}$, for which we define $\Sm_{\rm mix}(g)$ as a function of $g$ only. A natural way to do it in the general relativity framework is to fix a vector field and demand it to be timelike for all metrics (thus choosing an ``observer"), or -- dually -- consider a given spacelike distribution.

By Proposition \ref{3Sasakianiscritical} and Remark \ref{remarkbargvariations}, the metric of a Sasakian 3-structure is critical for the action \eqref{E-Jmix} with respect to volume preserving $g^\perp$- and $g^\top$-variations. By \cite{kashiwada1}, such metric is Einstein, and hence critical for all (also $g^\perp$- and $g^\top$-) variations preserving volume. Therefore, the metric of a Sasakian 3-structure is critical for the action \eqref{E-Jeps} with respect to both $g^\perp$- and $g^\top$-variations preserving volume, regardless of the value of the parameter $\eps$. Corollary \ref{EinsteinKilling} indicates the existence of metrics critical for \eqref{E-Jeps} with respect to volume preserving $g^\perp$- and $g^\top$-variations also in the 4-dimensional setting of general relativity. In the following section we obtain examples of metrics critical for \eqref{E-Jmix} with respect to arbitrary volume preserving variations.

\subsection{Codimension-one foliations} \label{sectionFoliations}

In this section we consider the action \eqref{E-Jmix}, where $\widetilde{\cal D}$ is tangent to a codimension-one foliation. We find metrics critical for all open, relatively compact subsets $\Omega \subset M$, with respect to both ${\bar g}^\perp$- and ${\bar g}^\top$-variations, as well as general variations of the metric preserving the volume of $\Omega$.

Let $\cal F$ be a codimension one foliation tangent to the distribution $\widetilde{\cal D}$. Let $h_{sc}$ be the scalar second fundamental form, and $A_N$ the Weingarten operator of ${\calf}$, we define the functions $\tau_i = \tr A_N^i\ (i\ge0)$. We~have $T=0=\tilde T$ and
\[
 h_{sc}(X,Y)=\epsilon_{N}\,g(\nabla_X\,Y,\,N),\quad A_N(X)=-\nabla_X\,N,\quad
 (X,Y\in T\calf).
\]
We define the vector field $( \widetilde{\Div} \, A_N)^{\sharp}\in\mathfrak{X}_{\widetilde{\cal D}}$ by
the following equation:
\[
 g(( \widetilde{\Div} \, A_N)^{\sharp}, X) = ( \widetilde{\Div} \, A_N)(X), \quad  (X \in \mathfrak{X}_{\widetilde{\cal D}} ).
\]
Then we can formulate the following

\begin{prop}
Let $\widetilde{\cal D}$ be the distribution tangent to a codimension one foliation of a manifold $M^{n+1}$. A metric $g$ on $M$ is critical for the action \eqref{E-Jmix} with respect to ${\bar g}^\perp$-variations if and only if:
\begin{eqnarray}
\label{codimoneEL1}
&& \tau_1^{2}-\tau_2 = -\epsilon_{N}\Sm^*_{\rm mix}(\Omega,g)  ,\\
\label{codimoneEL2}
&& (\widetilde{\Div} A_N)^{\sharp} - {\nabla}^\top \tau_1 = 0,
\end{eqnarray}
where $\Sm^*_{\rm mix} = \epsilon_{N}\Ric_{N} -2 \epsilon_{N}(N(\tau_1)-\tau_2)$.

A metric $g$ on $M$ is critical for the action \eqref{E-Jmix} with respect to ${\bar g}^\top$-variations if and only if:
\begin{equation} \label{codimoneEL3}
\nabla_N h_{sc} {-}\tau_1 h_{sc} = \frac12\big(2\,\epsilon_{N}(N(\tau_1){-}\tau_1^2)
 +\epsilon_{N}(\tau_1^{2}{-}\tau_2) - \widetilde{\Sm}^*_{\rm mix}(\Omega,g)\big)\, g^\top , 
\end{equation}
where $\widetilde{\Sm}^*_{\rm mix} = \epsilon_{N}\Ric_{N} - \frac{2}{n} \,\Div \tilde H$.
\end{prop}

\begin{proof}
Equations \eqref{codimoneEL1} and \eqref{codimoneEL3} were obtained from 
\eqref{E-main-0i} and \eqref{E-main-0iii} in \cite{rz-1} (in a similar way as in Section \ref{sectionFlows}).

For $X \in \widetilde{\cal D}$ and $N \in {\cal D}$ we have
\begin{eqnarray*}
 2\,\<\tilde\alpha, H \>(X,N) \eq \tau_1\,g(X, \tilde H), \\
 2\,{\rm Sym}(H^{\flat} \otimes {\tilde H}^{\flat})(X,N) \eq \tau_1\,g(X, \tilde H), \\
 2\,{\tilde \delta}_H (X,N) \eq X (\tau_1), \\
 2\,\Lambda_{\alpha, \tilde\alpha}(X, N) \eq g(A_N(\tilde H), X), \\
 2\,(\Div \alpha)(X,N) \eq (\Div A_N)(X). \\
\end{eqnarray*}
Using the above equations and the fact that for all $X \in \widetilde{\cal D}$ we have
\begin{eqnarray*}
 &&(\Div A_N)(X) = \sum\nolimits_{\,\lambda}\, \epsilon_{\lambda} g(\nabla_{\lambda}(A_N(X)), e_\lambda)
 -\sum\nolimits_{\,\lambda} \epsilon_{\lambda} g(A_N(\nabla_{\lambda} X), e_{\lambda}) \\
 \eq \sum\nolimits_{\,a} \epsilon_{a}\,g(\nabla_{E_a}(A_N(X)), E_a) + \epsilon_{N}\, g(\nabla_{N}(A_N(X)), N)
 -\sum\nolimits_{\,a} \epsilon_{a}\, g(A_N (\nabla_{E_a} X), E_a ) \\
 \eq (\widetilde{\Div}\, A_N)(X) - g( A_N (X), \tilde H) ,
\end{eqnarray*}
we reduce \eqref{E-main-0ii} to \eqref{codimoneEL2}.
\end{proof}

For ${\bar g}^\perp$- and ${\bar g}^\top$-variations, constants $\Sm^*_{\rm mix}(\Omega,g)$ and $\widetilde{\Sm}^*_{\rm mix}(\Omega,g)$ that appear in the Euler-Lagrange equations are related. We can consider \eqref{codimoneEL1} and \eqref{codimoneEL3} together, and use Lemma 2.11 from \cite{rz-1} to obtain the following
\begin{prop}\label{L-trace-tau1}
Let $\widetilde{\cal D}$ be tangent to a codimension-one foliation of a pseudo-Riemannian space $(M^{n+1},g)$,
and let the unit normal field $N$ of $\calf$ be complete in a domain $\Omega$ of $M$. Then metric $g$ is critical for the action \eqref{E-Jmix} with respect to both ${\bar g}^\perp$- and ${\bar g}^\top$-variations if and only if
\begin{eqnarray}\label{E-RicNs1F-2}
&& \tau_1^{2}-\tau_2 = \Ric_{N}(\Omega,g) - 2\,\hat C,\qquad
 \nabla_N h_{sc} -\tau_1 h_{sc} = \frac{\epsilon_N}{n}\,\hat C\,\tilde g, \\
 \label{E-RicNs1F-3}
&& (\widetilde{\Div} A_N)^{\sharp} - {\nabla}^\top \tau_1 = 0,
\end{eqnarray}
where $\hat C= (\Div\tilde H)(\Omega,g) \le0$. If ${\hat C}=0$ then $\tau_1=0$, otherwise
\begin{equation}\label{E-tau-N}
 \tau_1(t) = {|\hat C|}^{1/2}\,\Big(1-\frac{2({|\hat C|}^{1/2}-\tau_1^0)}
 {({|\hat C|}^{1/2}\!+\tau_1^0)\,e^{-2\,t\,{|\hat C|}^{1/2}}\!+{|\hat C|}^{1/2}\!-\tau_1^0}\Big),
  \ \ \tau_1(0)=\tau_1^0\in [-{|\hat C|}^{\frac12},{|\hat C|}^{\frac12}] .
\end{equation}
\end{prop}

Codimension-one foliations admit \textit{biregular foliated coordinates} $(x_0,\ldots,x_n)$, see \cite[Section~5.1]{cc1},
i.e., the leaves are the level sets $\{x_0=c\}$ and $N$-curves are given by $\{x_i=c_i\ (i>0)\}$.
From now on, we assume that a foliated pseudo-Riemannian manifold $(M, \calf, g)$ admits \textit{orthogonal bire\-gular foliated coordinates}
$($hence, $g_{ij}=0$ for $i\ne j)$,
then
 $g=g_{00}\,dx_0^2+\sum\nolimits_{i>0}g_{ii}\,(dx_i)^2$.
Denote by
$g_{ii,\mu}$ the derivative of $g_{ii}$ in the $\partial_\mu$-direction. We adopt the convention $\mu \in \{0 , \ldots , n \}$, $i,j \in \{1 , \ldots , n \}$.
We have
$g_{00}=\epsilon_N |g_{00}|$ and $g_{ii}=\epsilon_i |g_{ii}|$.

\begin{lem}\label{lem-biregG}
For a pseudo-Riemannian metric
in orthogonal
biregular foliated coordinates of a codimension-one foliation
on $(M,g)$, one has
\begin{eqnarray*}
 N \eq \partial_0/\sqrt{|g_{00}|}\quad\mbox{\rm(the unit normal)},\\
 \Gamma^j_{i0} \eq (1/2)\,\delta_i^j\,g_{ii,0}/g_{ii},\quad
 \Gamma^{i}_{00} = -(1/2)\,g_{00,i}/g_{ii},\quad
 \Gamma^{0}_{ij} = -\delta_{ij}\,g_{ii,0}/(2\,g_{00}),\\
 h_{ij}\eq
\Gamma^0_{ij}\sqrt{g_{00}}=
 -\frac12\,\epsilon_N\,\delta_{ij}\,g_{ii,0}/\sqrt{|g_{00}|}
 \quad\mbox{\rm(the second fundamental form)},\\
 A^j_{i}\eq
-\Gamma^j_{i0}/\sqrt{|g_{00}|}=
 -\frac{1}{2\,\sqrt{|g_{00}|}}\,\delta_{i}^{j}\,\frac{g_{ii,0}}{g_{ii}}
 \quad\mbox{\rm(the Weingarten operator)},\\
 \tau_1 \eq -\frac{1}{2\sqrt{|g_{00}|}}\sum\nolimits_{\,i>0}\,\frac{g_{i i,0}}{g_{i i}},\quad
 \tau_2=\frac{1}{4\,|g_{00}|}\sum\nolimits_{\,i>0}\,\Big(\frac{g_{i i,0}}{g_{i i}}\Big)^2,\quad \mbox{\rm etc}.
\end{eqnarray*}
Using the above, one can obtain
\begin{equation}\label{E-Qi}
  (\nabla_N\,h_{sc})_{ii} = -\frac{\epsilon_N}{2\,|g_{00}|}\,\big(g_{ii,00}
 - \frac12\,g_{ii,0}(\log|g_{00}|)_{,\,0} - (g_{ii,0})^2/g_{ii} \big).
\end{equation}
and, for $i = 1, \ldots, n$:
\begin{equation}
(\widetilde{\Div} A_N)(\partial_i) = \partial_i \big(- \frac{1}{2 \sqrt{ |g_{00} | }} \cdot \frac{g_{ii, \, 0}}{g_{ii}} \big) - \frac{1}{2 \sqrt{ |g_{00} | }} \cdot \frac{g_{ii, \, 0}}{g_{ii}} \cdot \sum_{a>0} \Gamma^{a}_{ai} + \frac{1}{2 \sqrt{ |g_{00} | }} \cdot \sum_{a>0} \Gamma^{a}_{ai} \frac{g_{aa, \, 0}}{g_{aa}},
\end{equation}
where
\begin{equation} \label{Gammaaai}
\Gamma^{a}_{ai} = \frac{1}{2} \cdot \frac{g_{aa, \, i}}{g_{aa}}.
\end{equation}
\end{lem}

\begin{lem}
Let $\calf$ be a codimension-one foliation of a pseudo-Riemannian manifold $(M,g)$ tangent to $\widetilde{\cal D}$, with a unit normal field $N$ complete in a domain $\Omega$. Let there exist global orthogonal biregular foliated coordinates $(x_0,x_1,\ldots, x_n)$, with the leaves of $\calf$
given by $\{x_0=c\}$, and
 $g$ of the form
\begin{equation}\label{E-g1122}
 g_{ii}=\epsilon_i\,f_i(x_1,\ldots x_n)\,e^{\,-2\int \sqrt{|g_{00}|}\,y_i(t , x_1 , \ldots , x_n )\,{\rm d}\,t},\quad i=1 , \ldots,  n ,
\end{equation}
where $f_i\ (i=1,\ldots n)$ are positive functions.

Then \eqref{E-RicNs1F-2} can be written as the system of two equations
\begin{equation}\label{sigma2const}
\tau_1^2 - \tau_2 = \Ric_{N}(\Omega,g) - 2\,\hat C
\end{equation}
and
\begin{equation}\label{odetauconst}
\partial_0 y_i -\tau_{1} \sqrt{|g_{00} |}\,y_i  -\frac{1}{n}\,\hat C\,\sqrt{|g_{00} |} = 0,
\quad i=1,\ldots ,n
\end{equation}
where $\tau_1 = y_1 + \ldots + y_n$ is given by \eqref{E-tau-N}, $\tau_2 =  y_1^2 + \ldots + y_n^2 $, 
$\hat C \leq 0$ is as in Proposition \ref{L-trace-tau1} and $g_{00}$ is a smooth function of constant sign.

Also, \eqref{E-RicNs1F-3} takes the following form:
\begin{equation}\label{E-main-iii-bifoliated}
 \partial_i y_i + y_i\sum\nolimits_{\,a>0}\Gamma^{a}_{ai}
 -\sum\nolimits_{\,a>0} y_a\Gamma^{a}_{ai} -\partial_i \sum\nolimits_{\,a>0} y_a =0, \quad i=1, \ldots, n,
\end{equation}
which can be written equivalently as
\begin{equation}\label{E-main-iii-bifoliated2}
 \partial_i\sum\nolimits_{\,a>0,\,a\neq i} y_a +\sum\nolimits_{\,a>0,\,a\neq i}\Gamma^{a}_{ai}(y_i - y_a) = 0,
 \quad i=1, \ldots, n.
\end{equation}
\end{lem}

\begin{proof}
This follows from a straightforward computation, using \eqref{E-g1122} and Lemma \ref{lem-biregG}.
\end{proof}

In \cite{rz-1}, under the assumption of the existence of a global biregular orthogonal coordinate system, the system of equations \eqref{E-RicNs1F-2} was fully solved for $n=2$ and examples of its solutions in higher dimensions were obtained - in particular, making $\widetilde{\cal D}$ minimal or totally umbilical. In general, metrics satisfying \eqref{E-RicNs1F-2} and \eqref{E-RicNs1F-3} are critical only for a particular set $\Omega$, since terms $(\Div\tilde H)(\Omega,g)$ and $\Ric_{N}(\Omega,g)$ explicitly appear in the Euler-Lagrange equations. However, for $\widetilde{\cal D}$ totally umbilical or minimal, and with constant principal curvatures, we can find metrics that are critical (with respect to both ${\bar g}^\perp$- and ${\bar g}^\top$-variations) for all sets $\Omega \subset M$.

\begin{prop}
Let $\widetilde{\cal D}$ be tangent to a codimension one foliation $\calf$ and let there exist global orthogonal biregular foliated coordinates $(x_0,x_1,\ldots, x_n)$, with the leaves of $\calf$
given by $\{x_0=c\}$.
Let ${\hat C} < 0$ and let $y = \pm \frac{1}{n} \sqrt{ | {\hat C} | }$. Then there exist metrics of the form \eqref{E-g1122} with $y_i = y$ for all $i = 1 , \ldots , n$ that are critical for the action \eqref{E-Jmix} with respect to both ${\bar g}^\perp$- and ${\bar g}^\top$-variations, for all $\Omega \subset M$.
\end{prop}

\begin{proof}
By the assumption about $y$, \eqref{odetauconst} is satisfied and \eqref{sigma2const} takes the 
form
$\Ric_{N}(\Omega,g) = - \frac{n+1}{n} | {\hat C} |$. Hence, the metric is critical for all $\Omega$ if and only if it satisfies \eqref{E-RicNs1F-3} together with the following, point-wise condition:
\begin{equation} \label{RicNumb}
\Ric_{N} = - \frac{n+1}{n} | {\hat C} | .
\end{equation}

From \eqref{E-main-iii-bifoliated2} and $\partial_i y=0$ it follows that \eqref{E-RicNs1F-3} holds for any metric of the form \eqref{E-g1122}. We have an identity \cite{rz-1}
\begin{equation}\label{eq-ran1}
 \Ric_{N} = N(\tau_1)-\tau_2 +\Div(\nabla_{N}\,N),
\end{equation}
that in our case takes form
\begin{equation}
\Ric_{N} = \Div(\nabla_{N}\,N) - \frac{1}{n} | {\hat C} |,
\end{equation}
and from Lemma \ref{lem-biregG} we can obtain the following formula:
\begin{equation} \label{divNN}
 \Div(\nabla_NN) = \sum\nolimits_{\,i>0}\big(\epsilon_i g_{ii} Q_{i,i} +\frac12\big(\epsilon_N g_{00,i}
 +\sum\nolimits_{\,j>0} \epsilon_j\,g_{jj,i} \big)\,Q_i\big),\ {\rm where}\ Q_i=-\frac1{2|g_{00}|}\,\frac{g_{00,i}}{g_{ii}}.
\end{equation}
Hence, to solve \eqref{RicNumb}, we need to find a solution of
\begin{equation} \label{E-Z0}
\Div\nabla_NN = - | {\hat C} |
\end{equation}
of the form \eqref{E-g1122} with $y_i =y$ for all $i = 1 , \ldots , n$.

To show that solutions of \eqref{E-Z0} exist, we consider the case $n=2$ and for simplicity additionally assume that
\[
 f_a=1,\quad \epsilon_a=1\ \ (a=1,2),\quad
 \epsilon_N=1,\quad
 g_{00}=w(x_1,x_2) T( x_0 ),
\]
for some functions $w>0$ and $T>0$.
Then $g_{11}=g_{22}$ are functions of $t$ and $w$.
Hence, equation \eqref{E-Z0} yields an elliptic PDE (with parameter $t$) for $w$:
\begin{eqnarray*}
 && \Delta w +f(t,w)\,\<\nabla w,\nabla w\> =  - | {\hat C} | ,\quad {\rm where}\\
 && f(t,w)= \frac12\,T (t)\, {e}^{\sqrt{|\hat C|\,w}\int\!\sqrt{T(t)}\,{\rm d}t}
 +\frac{\sqrt{C}\int\!\sqrt{T(t)}\,{\rm d}t} {2\,{w^{1/2}}}\Big(1+\frac{1}{2\,w\,T(t)}\Big)
 -\frac1{w}\,.
\end{eqnarray*}
The substitution $u=\int\frac{{\rm d}w}{F(w)}$ with $F(w)=e^{\,\int f(w)\,{\rm d}w}$ leads to the
Poisson's equation
 $\Delta u = - | {\hat C} |$.
\newline
Hence, $u=u_0(x_1,x_2) - \frac12  | {\hat C} | (x_1^2+x_2^2)$, where $u_0$ is a harmonic function, yields a solution. This approach can be generalized to higher dimensions.
\end{proof}

\begin{prop}
Let $\widetilde{\cal D}$ be tangent to a codimension one foliation $\calf$ and let there exist global orthogonal biregular foliated coordinates $(x_0,x_1,\ldots, x_n)$, with the leaves of $\calf$
given by $\{x_0=c\}$.
Then for any set of constants $(y_i)_{1 \leq i \leq n}$ such that $y_1 + \ldots + y_n = 0$ there exists a metric critical for the action~\eqref{E-Jmix} with respect to both ${\bar g}^\perp$- and ${\bar g}^\top$-variations, for all sets $\Omega \subset M$, such that $\widetilde{\cal D}$ is minimal, with principle curvatures $(y_i)_{1 \leq i \leq n}$.
\end{prop}

\begin{proof}
Let $g$ be of the form \eqref{E-g1122} with $y_i$ as assumed, and let $g_{00 , i} =0$ for $i \in {1 , \ldots , n}$. Then $0 = \Div \nabla_N N = {\hat C}$ and from \eqref{eq-ran1} it follows that \eqref{sigma2const} is satisfied as an identity. Equation \eqref{odetauconst} is also satisfied, because all its terms vanish. To finish the proof, take in \eqref{E-g1122} any set of positive functions $f_i$ such that \eqref{E-main-iii-bifoliated2} is satisfied. To see that it is possible, note that \eqref{E-main-iii-bifoliated2} reads that for all $i>0$, at every point of $M$, the vector $( \Gamma^1_{1i} , \ldots , \Gamma^n_{ni} )$ is orthogonal in $\RR^n$ to a constant vector $( y_1 - y_i , \ldots , y_n - y_i)$, and with the assumption $g_{00 , i} =0$, from \eqref{Gammaaai} we obtain $\Gamma^a_{ai} = \partial_i \log f_a$.
\end{proof}

The Euler-Lagrange equations for arbitrary volume preserving variations are more flexible and 
do not depend on $\Omega$. In the case of codimension one foliations, they admit a number of solutions.

\begin{prop} \label{codimonevolpreserving}
Let $\widetilde{\cal D}$ be tangent to a codimension-one foliation of a pseudo-Riemannian space $(M^{n+1},g)$, $n >1$, and let the unit normal field $N$ of $\calf$ be complete in a domain $\Omega$ of $M$. Then metric $g$ is critical for the action \eqref{E-Jmix} with respect to all volume preserving variations if and only if
\begin{eqnarray} \label{codim1folgenvar}
&& \nabla_N h_{sc} {-}\tau_1 h_{sc} + \frac{ \epsilon_N (\tau_1^{2}-\tau_2)}{n-1} g^\top =0 ,\\
 \label{E-RicNs1F-3gen}
&& (\widetilde{\Div} A_N)^{\sharp} - {\nabla}^\top \tau_1 = 0.
\end{eqnarray}
\end{prop}

\begin{proof}
For arbitrary volume preserving variations equation \eqref{codimoneEL2} remains unchanged, and the other Euler-Lagrange equations read:
\begin{eqnarray}\label{E-RicNs0Fgen0}
 &&\tau_1^{2}-\tau_2 = -\epsilon_{N} \lambda 
 ,\\
\label{E-RicNs1Fgen0}
 &&\hskip-13mm \nabla_N h_{sc} {-}\tau_1 h_{sc} = \frac12\big(2\,\epsilon_{N}(N(\tau_1){-}\tau_1^2)
 +\epsilon_{N}(\tau_1^{2}{-}\tau_2) - \lambda \big)\, g^\top
 ,
\end{eqnarray}
with $\lambda \in C^\infty(M)$, and for $n \neq 1$ simplify to
\begin{eqnarray}\label{E-RicNs0Fgen}
 &&\tau_1^{2}-\tau_2 = -\epsilon_{N} \lambda 
 ,\\
\label{E-RicNs1Fgen}
 &&\hskip-13mm \nabla_N h_{sc} {-}\tau_1 h_{sc} = \frac{1}{n-1} \lambda \, g^\top .
\end{eqnarray}
Since $\lambda$ is an arbitrary function, without any loss of generality the two equations above can be joined and written as \eqref{codim1folgenvar}.
\end{proof}

The proof of the following Lemma is a straightforward computation.

\begin{lem}
Let $g$ be of the form \eqref{E-g1122}. Then \eqref{codim1folgenvar} becomes system of the following equations:
\begin{equation} \label{codim1folgenvarg1122}
\partial_0 y_i =  \sqrt{ | g_{00} | } y_i ( \sum_{a=1}^n y_a ) - \frac{1}{n-1}\sqrt{ | g_{00} | } ( (\sum_{a=1}^n y_a)^2 - \sum_{a=1}^n y_a^2  ) , \quad i =1 , \ldots , n
\end{equation}
and \eqref{E-RicNs1F-3} takes the form \eqref{E-main-iii-bifoliated}, which is equivalent to \eqref{E-main-iii-bifoliated2}.
\end{lem}

The above lemma indicates the existence of multiple critical metrics. For example, for $n=2$ and $g_{00} \equiv 1$ we have the following solution of \eqref{codim1folgenvarg1122}:
\begin{eqnarray*}
y_1 \eq - \sqrt{ c_1 } \coth ( \sqrt{ c_1 } x_0 + \sqrt{ c_1 } c_2 )   \\
y_2 \eq - \sqrt{ c_1 } \tanh ( \sqrt{ c_1 } x_0 + \sqrt{ c_1 } c_2 )   ,
\end{eqnarray*}
with $c_1$, $c_2$ being constants, or - if we want to consider also \eqref{E-RicNs1F-3} - functions independent on $x_0$ such that \eqref{E-main-iii-bifoliated2} is satisfied.

Using Proposition \ref{codimonevolpreserving}, we can fully characterize metrics with totally umbilical or minimal $\widetilde{\cal D}$, that are critical for \eqref{E-Jmix} with respect to all volume preserving variations.

\begin{cor}
Let $g$ be a metric on $M^{n+1}$, with $n>1$, critical for the action \eqref{E-Jmix} with respect to all volume preserving variations and let $\widetilde{\cal D}$ be minimal and tangent to a codimension one foliation. Then $\widetilde{\cal D}$ is totally geodesic.
\end{cor}

\begin{proof}
Assuming $\tau_1 =0$ and tracing \eqref{codim1folgenvar}, we obtain $\tau_2 =0$.
\end{proof}

\begin{cor}
Let $g$ be a metric on $M^{n+1}$, with $n>1$, critical for the action \eqref{E-Jmix} with respect to all volume preserving variations and let $\widetilde{\cal D}$ be totally umbilical, tangent to a codimension one foliation, with unit normal field $N$. 
Then $\tau_1$ is constant on $M$.
\end{cor}

\begin{proof}
Let $\tau_1 = ny$.
From \eqref{codim1folgenvar} we obtain $N(y)=0$. For totally umbilical foliations \eqref{E-RicNs1F-3gen} reduces to
\[
(n-1)\,  \nabla^\top y = 0
\]
and the claim follows.
\end{proof}

\textbf{Acknowledgments}.
This paper has been written during the second author's stay at a postdoctoral fellowship in Department of Mathematics at University of Haifa.
The second author would like to thank the first author and the University of Haifa, for their hospitality.

\end{document}